\documentclass[11pt]{article}
\usepackage{appendix}
\usepackage{mathtools}
\usepackage[T1]{fontenc}
\usepackage{amsfonts}
\usepackage{amsmath}
\usepackage{amssymb}
\usepackage{amsthm}
\usepackage{bbm}
\usepackage{bm}
\usepackage{mathrsfs}
\usepackage{color}
\usepackage{pdfsync}
\usepackage{placeins}
\usepackage{enumitem}
\usepackage[colorlinks=true, linkcolor=blue, citecolor=blue, urlcolor=blue]{hyperref}
\hypersetup{
  urlcolor = black,
  pdfauthor = {Alberto Gonzalez-Sanz, Eustasio del Barrio and Marcel Nutz},
  pdfkeywords = {Optimal Transport; Quadratic Regularization; Sample Complexity; Central Limit Theorem},
  pdftitle = {Sample Complexity of Quadratically Regularized Optimal Transport},
  pdfsubject = {Sample Complexity of Quadratically Regularized Optimal Transport},
  pdfpagemode = UseNone
}

\usepackage{tikz}

\DeclareMathOperator*{\spt}{spt}

\DeclareMathOperator*{\Var}{Var}

\newcommand{\RR}{\mathbb{R}}
\newcommand{\R}{\RR}

\newcommand{\EE}{\mathbb{E}}
\newcommand{\PP}{\mathbb{P}}
\newcommand{\NN}{\mathbb{N}}

\newcommand{\eps}{\varepsilon}

\newcommand{\Banach}{\mathcal{B}_\oplus}

\newcommand{\1}{{\mathbf 1}}
\usepackage[margin=31 mm]{geometry}
\newcommand{\mykill}[1]{}

\usepackage[capitalize, noabbrev]{cleveref}
\crefname{assumption}{Assumption}{Assumptions}
\crefname{equation}{}{} %
\theoremstyle{plain}
\newtheorem{theorem}{Theorem}[section]
\newtheorem{proposition}[theorem]{Proposition}
\newtheorem{lemma}[theorem]{Lemma}
\newtheorem{corollary}[theorem]{Corollary}

\theoremstyle{definition}

\newtheorem{remark}[theorem]{Remark}

\newtheorem{assumption}[theorem]{Assumption}

\renewenvironment{thebibliography}[1]{%
\begin{oldthebibliography}{#1}%
\setlength{\baselineskip}{.9em}
\linespread{1}
\small
\setlength{\parskip}{.40ex}%
\setlength{\itemsep}{.30em}%
}%
{%
\end{oldthebibliography}%
}

\begin{document}

\title{\vspace{-2.5em} Sample Complexity of 
Quadratically Regularized\\Optimal Transport}
\date{\today}
\author{Alberto Gonz{\'a}lez-Sanz%
  \thanks{Department of Statistics, Columbia University, ag4855@columbia.edu} 
  \and  
 Eustasio del Barrio%
  \thanks{IMUVa, Universidad de Valladolid, eustasio.delbarrio@uva.es}
  \and  
  Marcel Nutz%
  \thanks{Departments of Statistics and Mathematics, Columbia University, mnutz@columbia.edu. Research supported by NSF Grants DMS-2106056, DMS-2407074.}
  }
  
\maketitle \vspace{-2em}
\begin{abstract}
It is well known that optimal transport suffers from the curse of dimensionality: when the prescribed marginals are approximated by i.i.d.\ samples, the convergence of the empirical optimal transport problem to the population counterpart slows exponentially with increasing dimension. Entropically regularized optimal transport (EOT) has become the standard bearer in many statistical applications as it avoids this curse. Indeed, EOT has parametric sample complexity, as has been shown in a series of works based on the smoothness of the EOT potentials or the strong concavity of the dual EOT problem. However, EOT produces full-support approximations to the (sparse) OT problem, leading to overspreading in applications, and is computationally unstable for small regularization parameters.

The most popular alternative is quadratically regularized optimal transport (QOT), which penalizes couplings by $L^2$ norm instead of relative entropy. QOT produces sparse approximations of OT and is computationally stable. However, its potentials are not smooth (do not belong to a Donsker class) and its dual problem is not strongly concave, hence QOT is often assumed to suffer from the curse of dimensionality. In this paper, we show that QOT nevertheless has parametric sample complexity. More precisely, we establish central limit theorems for its dual potentials, optimal couplings, and optimal costs. Our analysis is based on novel arguments that focus on the regularity of the support of the optimal QOT coupling. Specifically, we establish a Lipschitz property of its sections and leverage VC theory to bound its statistical complexity. Our analysis also leads to gradient estimates of independent interest, including $\mathcal{C}^{1,1}$~regularity of the population potentials.
\end{abstract}

 \vspace{0.5em}

{\small

}
 \vspace{.0em}

\section{Introduction}

The  optimal transport (OT) problem between compactly supported probability measures~\(P, Q\) on \(\mathbb{R}^d\) is
\begin{equation}
    \label{OT}
    \tag{OT}
    {\rm OT}(P,Q) := \min_{\pi \in \Pi(P,Q)} \int \frac{1}{2} \|x - y\|^2 \, d\pi(x,y),
\end{equation}
where \(\Pi(P,Q)\) denotes the set of couplings between \(P\) and \(Q\); i.e., probability measures on \(\mathbb{R}^d \times \mathbb{R}^d\) with marginals \((P,Q)\). A well-known limitation for its application in high-dimensional statistics and data science is the curse of dimensionality suffered by~\eqref{OT} when the marginals are approximated by samples; see \cite{ChewiNilesWeedRigollet.25} for a recent survey. 
Regularized optimal transport (ROT)  introduces a divergence penalty to \eqref{OT}, 
\begin{equation}
    \label{ROT}
    \tag{ROT}
    {\rm ROT}_\eps(P,Q) := \min_{\pi \in \Pi(P,Q)} \int \frac{1}{2} \|x - y\|^2 \, d\pi(x,y) + \eps \int \varphi\left( \frac{d\pi}{d(P \otimes Q)} \right) d(P \otimes Q),
\end{equation}
where $d\pi/d(P \otimes Q)$ is the Radon--Nikodym density of \(\pi\) w.r.t.\ the product measure \(P \otimes Q\), $\eps>0$ is the regularization parameter and \(\varphi\) is a convex function. A standard choice is the logarithmic entropy \(\varphi(x) = x \log x\), leading to the entropic optimal transport (EOT) problem. EOT admits efficient computation via Sinkhorn's algorithm \cite{Cuturi.2013.Neurips} and avoids the curse of dimensionality: when $(P,Q)$ are approximated by the empirical measures corresponding to i.i.d.\ samples, the empirical optimal costs, couplings, and dual potentials all converge to their population counterparts at the parametric rate \(n^{-1/2}\) \cite{genevay.2019.PMLR,MenaWeed.2019.Nips,delBarrioEtAl.2023.SIMODS}. In addition, central limit theorems for all three quantities have been established \cite{goldfeld.2024.statisticalinferenceregularizedoptimal,GonzalezSanz.2024.weaklimits,GonzalezSanz.2023.Beyond,delBarrioEtAl.2023.SIMODS}; see also the recent surveys \cite{balakrishnan.et.al.2025.survey,delbarrio.et.al.2025.survey}. 
On the flip side, using logarithmic entropy entails that the optimal coupling of~\eqref{ROT} has full support (equal to the support of $P\otimes Q$), a phenomenon known as overspreading since the true optimal coupling for~\eqref{OT} is sparse (given by Brenier's map) as soon as one marginal is absolutely continuous. Overspreading can lead to blurring in image processing~\cite{blondel18quadratic} or bias in manifold learning~\cite{zhang.2023.manifoldlearningsparseregularised}, for example.

Starting with \cite{Muzellec.2017.AAAI,blondel18quadratic,EssidSolomon.18}, quadratic regularization has emerged as the most popular alternative to logarithmic entropy. Indeed, the choice $\varphi(x) = \frac{1}{2}x^2$ gives rise to the squared $L^2$ norm (or $\chi^2$ divergence) as penalty  and leads to the quadratically regularized optimal transport (QOT) problem 
\begin{equation}
    \label{QOT}
    \tag{QOT}
    {\rm QOT}_\eps(P,Q) := \min_{\pi \in \Pi(P,Q)} \int \frac{1}{2}\|x - y\|^2 \, d\pi(x,y) + \frac{\eps}{2} \left\|\frac{d\pi}{d(P \otimes Q)}\right\|_{L^2(P \otimes Q)}^2.
\end{equation}
The optimal cost ${\rm QOT}_\eps(P,Q)$ approximates ${\rm OT}(P,Q)$ at rate \( \eps^{2/(d+2)} \) as $\eps\to0$ (see \cite{EcksteinNutz.22,GarrizmolinaElAl.2024}) and, in contrast to EOT, the support of the optimal coupling converges to the support of the unregularized optimal transport. This sparsity for small~$\eps$ has been observed empirically since the initial works (e.g., \cite{blondel18quadratic,EssidSolomon.18, Lorenz.2019}) and established theoretically more recently in~\cite{Nutz.24,WieselXu.24,GonzalezSanzNutz2024.Scalar}. A separate advantage is that EOT tends to be computationally unstable for small regularization parameters due to the occurrence of exponentially large/small values, as noted for instance in~\cite{LiGenevayYurochkinSolomon.2020.NIPS}. Such values do not occur for QOT. While QOT has long been effective in computational practice, we also mention the recent theoretical guarantee~\cite{GonzalezSanzNutzRiveros.25gradDesc} that gradient descent for the dual problem of QOT converges exponentially fast.

Denoting $\left( x \right)_+ =\max\{x,0\}$, the dual problem of \eqref{QOT} reads
\begin{multline}
    \label{DQOT}
    \sup_{(f, g)\in \mathcal C(\R^d)\times\mathcal C(\R^d)} \int f(x) \, dP(x) + \int g(y) \, dQ(y) \\
    - \frac{1}{2\eps} \int \left( f(x) + g(y) - \frac{\|x - y\|^2}{2} \right)_+^2 \, d(P \otimes Q)(x,y).
\end{multline}
Its optimizers \((f_\eps, g_\eps)\)  are called the QOT potentials and describe the density of the optimal coupling $\pi_\eps$ of \eqref{QOT} via
\begin{equation}
    \label{eq:primal-dual-intro}
    d\pi_\eps(x,y) = \frac{1}{\eps} \left( f_\eps(x) + g_\eps(y) - \frac{\|x - y\|^2}{2} \right)_+  \, d(P \otimes Q)(x,y).
\end{equation}
We observe that the objective function in the dual problem~\eqref{DQOT} is not strongly (or even strictly) concave. Moreover, the regularity of \((f_\eps, g_\eps)\) does not generally surpass that of the unregularized OT potentials (e.g., \cite{GonzalezSanzNutz.24b}). For general divergences $\varphi$, the dual has a form similar to~\eqref{DQOT} but with $\frac12(\cdot)_+^2$ replaced by the conjugate $\psi(s)=\sup_{t\geq 0} \{ st-\varphi(t)\}$ of $\varphi$.
For the particular choice of logarithmic entropy, the dual objective is strongly concave and moreover the potentials inherit the smoothness of the transport cost $\frac{1}{2}\|x - y\|^2$. This is crucial for the aforementioned works on sample complexity and central limit theorems for EOT, which can be divided into two approaches: the first, starting with \cite{genevay.2019.PMLR,MenaWeed.2019.Nips}, is based on empirical process theory and requires uniform regularity of the empirical dual potentials; the second, starting with \cite{delBarrioEtAl.2023.SIMODS,rigollet2022samplecomplexityentropicoptimal}, exploits the strong concavity of the dual EOT problem (see \cite{balakrishnan.et.al.2025.survey,delbarrio.et.al.2025.survey,gonzalezsanz.2025.sparseregularizedoptimaltransport} for more detailed reviews). Analyzing more general ROT---where the dual is not strongly concave and the potentials are not smooth---through the same approaches, one obtains results as in \cite{BayraktarEckstein.2025.BJ} which are consistent with ROT suffering from the curse of dimensionality. By contrast, and maybe surprisingly, the recent work \cite{gonzalezsanz.2025.sparseregularizedoptimaltransport} establishes central limit theorems for a class of ROT problems that are not smooth or strongly concave. Its main condition on the divergence is that the conjugate $\psi(s)=\sup_{t\geq 0} \{ st-\varphi(t)\}$ be $\mathcal{C}^2$. 
This condition fails for QOT, arguably the most interesting example for ROT, where $\psi$ is $\mathcal{C}^{1,1}$ but not~$\mathcal{C}^2$. For QOT, the (weak) second derivative is an indicator function rather than a continuous function, and this difference derails the core arguments in \cite{gonzalezsanz.2025.sparseregularizedoptimaltransport}.

The present work establishes central limit theorems for the dual potentials, optimal costs, and couplings of QOT (\cref{th:CLT.potentials,th:CLT.cost,th:CLT.plans}). While their assertions are analogous to the ones in \cite{gonzalezsanz.2025.sparseregularizedoptimaltransport}, the derivations are substantially different. In \cite{gonzalezsanz.2025.sparseregularizedoptimaltransport}, it is assumed that the conjugate $\psi$ is $\mathcal{C}^2$, and this allows for soft arguments---the marginal measures are very general and the transport cost is any $\mathcal{C}^1$ function. For QOT, this assumption fails and finer structures are needed. We focus on quadratic transport cost $\frac{1}{2}\|x - y\|^2$ and marginals with convex supports and bounded Lebesgue densities; these structures are crucial for our regularity results, which are also some of our main technical contributions. 

One of those is a Lipschitz regularity property of the sections $\mathcal{S}_x$ of the support of the optimal population coupling. Representing $\mathcal{S}_x$ as the zero-sublevel set of a convex function, we consider more generally the $\beta$-sublevel sets which correspond to a thickening or shrinking of $\mathcal{S}_x$, and establish a Lipschitz estimate w.r.t.\ $\beta$ (see \cref{prop:bound-Lipschitz}). Both the result and the broader approach  are novel to the best of our knowledge. The result is crucial in several steps of the analysis, such as establishing the Fréchet differentiability of the first-order condition for the population potentials (\cref{lemma:frechet}). Another consequence is the $\mathcal{C}^{1,1}$-regularity of the QOT potentials (\cref{co:potentialGradientsAreLipschitz}), which is of independent interest and links to Caffarelli's regularity theory for~\eqref{OT}.

Another key innovation is to employ Vapnik--Chervonenkis (VC) theory. Specifically, we use VC theory to bound the statistical complexity of the sections $\mathcal{S}_x$ of the support of the optimal coupling (and more generally its thickenings or shrinkings). We package this bound into an abstract tool, \cref{pr:VCtool}, which is used several times throughout the proof of the central limit theorems; see also \cref{se:proofMethodology} for an overview of the proof methodology. At a high level, this VC-theoretic bound acts as a substitute for the missing regularity of~$\psi$. To the best of our knowledge, VC theory has not been used previously in the sample complexity analysis of optimal transport or regularized optimal transport.

The remainder of the paper is organized as follows. \Cref{se:SetupBackground} details the setting, notation and assumptions, and summarizes background facts about QOT for later reference. \Cref{se:TheCLTs} states our three central limit theorems---\cref{th:CLT.potentials,th:CLT.cost,th:CLT.plans} for the dual potentials, the optimal costs, and the optimal couplings---followed by an overview of the proof methodology. In \cref{se:RegularityEstimates} we provide regularity results for the potentials and the optimal supports. These results are some of our main technical contributions; they are of independent interest in addition to being key ingredients for proving the central limit theorems. After these preparations, \cref{se:proof.CLT.potentials} proves the central limit theorem for the potentials, \cref{th:CLT.potentials}. \Cref{Section:CLT-plans-and-cost} proves the remaining central limit theorems for the optimal costs and couplings (\cref{th:CLT.cost,th:CLT.plans}), which follow from \cref{th:CLT.potentials} and some additional work. Finally, \cref{se:omittedProofs} details omitted proofs and references for the background facts in \cref{se:SetupBackground}.

\section{Setup and Background}\label{se:SetupBackground}

This section details our setup and summarizes (mostly) known background material for later reference. Proofs are deferred to~\cref{se:omittedProofs}.

Given $\mathcal{X}\subset\R^d$, we set $\|\mathcal{X}\|_\infty:=\sup_{x\in\mathcal{X}} \|x\|$. We write $\mathcal{C}(\mathcal{X})$ for the space of continuous real-valued functions on~$\mathcal{X}$ and $\mathcal{C}^{k,\alpha}(\mathcal{X})$ for the functions with $k$ continuous derivatives whose $k$-th derivatives are Hölder continuous of order~$\alpha$, $0< \alpha\leq 1$. %
A random sequence $U_n$ in a separable Banach space $\mathcal{B}$ converges weakly to $U$ if $\EE[f(U_n) ] \to \EE[f(U) ]$ for every bounded and continuous function $f: \mathcal{B}\to \R$; in that case, we write 
 $U_n \overset{\mathcal{B}}{ \rightsquigarrow} U$. When $\mathcal{B}=\R$, we also write  $\xrightarrow{w}$ instead of $ \overset{\R}{ \rightsquigarrow}$. %

\subsection{Population setting}

Throughout, we impose the following conditions on the given (population) marginal measures. 

\begin{assumption}\label{Assumption:reg}
    The probability measures $P,Q$ have compact, convex supports $\mathcal{X},\mathcal{Y}$ in~$\R^d$ and admit bounded densities 
    w.r.t.~the Lebesgue measure~$\mathcal{L}_d$. 
\end{assumption}

Given $(f,g)\in\mathcal{C}(\mathcal{X})\times \mathcal{C}(\mathcal{Y})$, we denote by $f\oplus g$ the function $(x,y)\mapsto f(x)+g(y)$. As we will often be interested only in $f\oplus g$ rather than $f$ and $g$ individually, it is useful to define the equivalence relation $(f,g)\sim_\oplus (f',g')$ iff $f\oplus g=f'\oplus g'$, which is equivalent to the existence of $a\in\R$ with $f=f'+a$ and $g=g'-a$. We define the quotient spaces
\begin{align}\label{eq:defBanach}
  \mathcal{B}_{\oplus} = (\mathcal{C}(\mathcal{X})\times \mathcal{C}(\mathcal{Y}))/\sim_\oplus, 
  \qquad 
  \mathcal{B}^{0,1}_{\oplus} = (\mathcal{C}^{0,1}(\mathcal{X})\times \mathcal{C}^{0,1}(\mathcal{Y}))/\sim_\oplus
\end{align}
and denote by $\|\cdot\|_\oplus$ and $\|\cdot\|_{\oplus,1}$ the quotient  norms on $\mathcal{B}_\oplus$ and $\mathcal{B}^{0,1}_\oplus$, respectively.

\begin{lemma}[Population potentials]\label{le:population.potentials}
    Let $P$ and $Q$ satisfy \cref{Assumption:reg}. 
    \begin{enumerate}
    \item There exists a unique pair $(f_\eps,g_\eps)\in\mathcal{B}_\oplus$, called the population potentials, solving the dual problem~\eqref{DQOT}. 
    \item The functions $(\varphi_\eps,\psi_\eps)\in\mathcal{B}_\oplus$ defined by 
    \begin{align}\label{eq:convexPopPotentials}
        \varphi_\eps(x)=\frac{\|x\|^2}{2}-f_\eps(x), \qquad \psi_\eps(y)=\frac{\|y\|^2}{2}-g_\eps(y)
    \end{align}
    are called the (transformed) potentials. They are uniquely characterized by the first-order optimality condition
    \begin{equation}\label{population_optimality}
    \begin{cases}
      \eps=  \int  \left( \langle x, y \rangle - \varphi_\eps(x)- \psi_\eps(y) \right)_+ dQ(y) & \text{ for all } x\in \mathcal{X},\\  
      \eps=  \int  \left( \langle x, y \rangle - \varphi_\eps(x)- \psi_\eps(y)\right)_+ dP(x) &\text{ for all } y\in \mathcal{Y}.
    \end{cases}
    \end{equation}    
    Moreover, $\varphi_\eps$ and $\psi_\eps$ are convex and continuously differentiable, with gradients
    \begin{align}\label{eq:population-gradient-formula}
    \nabla \varphi_\eps(x)= \frac{\int_{\mathcal{S}_x} y dQ(y)}{Q(\mathcal{S}_x)} \quad \text{and}\quad \nabla \psi_\eps(y)= \frac{\int_{\mathcal{T}_y} x dP(x)}{P(\mathcal{T}_y)}, \quad\mbox{where}
    \end{align}
    \begin{align}\label{eq:sections}
    \mathcal{S}_x= \{y\in \mathcal{Y}: \varphi_\eps(x)+ \psi_\eps(y)\leq \langle x, y \rangle\}, \quad  \mathcal{T}_y= \{x\in \mathcal{X}: \varphi_\eps(x)+ \psi_\eps(y)\leq \langle x, y \rangle\} .
    \end{align}  
    In particular, $\varphi_\eps$ and $\psi_\eps$ are Lipschitz with constants $\|\mathcal{Y}\|_\infty$  and $\|\mathcal{X}\|_\infty$, respectively.
    
    \item The sets $\mathcal{S}_x$ and  $\mathcal{T}_y$ 
    are convex, have nonempty interior, boundaries given by
    \begin{equation}\label{eq:sectionsBoundaries}
    \partial{\mathcal{S}}_x=\{y\in\mathcal{Y}:\,\varphi_\eps(x)+\psi_\eps(y)=\langle x,y \rangle\}, \quad  \partial{\mathcal{T}}_y=\{x\in\mathcal{X}:\,\varphi_\eps(x)+\psi_\eps(y)=\langle x,y \rangle\}
    \end{equation}
    and there exists $\delta>0$ such that
    \begin{equation}\label{eq:sectionsBounds}
    \delta\leq  Q( \mathcal{S}_x) \quad \text{and}\quad \delta\leq   P( \mathcal{T}_y)  \quad \text{for all } x\in \mathcal{X}, \ y\in \mathcal{Y}.    
    \end{equation}
    \item The primal problem~\eqref{QOT} has a unique solution $\pi_\eps\in\Pi(P,Q)$. It is related to the potentials by
    \begin{equation}
    \label{eq:primal-dual}
    \frac{d\pi_\eps}{d(P \otimes Q)}(x,y) = \frac{1}{\eps} \left( f_\eps(x) + g_\eps(y) - \frac{\|x - y\|^2}{2} \right)_+ = \frac{1}{\eps} \left( \langle x, y \rangle - \varphi_\eps(x)- \psi_\eps(y) \right)_+
\end{equation}
    and its support is
        $\spt \pi_\eps = \{(x,y)\in \mathcal{X}\times\mathcal{Y}: \varphi_\eps(x)+ \psi_\eps(y)\leq \langle x, y \rangle\}$,
    so that $\mathcal{S}_x$ and~$\mathcal{T}_y$ can be interpreted as the sections of the support.
    \end{enumerate}
\end{lemma}

We mention that the continuous differentiability of the population potentials will be improved to $\mathcal{C}^{1,1}$-regularity in \cref{co:potentialGradientsAreLipschitz}. We will find it convenient to use both $(f_\eps,g_\eps)$ and $(\varphi_\eps,\psi_\eps)$, and refer to either as potentials. The former pair helps to make the main results more comparable to the related literature whereas the proofs use the latter, in order to benefit from its convexity properties. 

\subsection{Empirical setting}
Next, we turn to the empirical setting. Let $X_1, \dots, X_n\overset{\rm iid}{\sim} P$ and $Y_1, \dots, Y_n\overset{\rm iid}{\sim} Q$ be independent samples, and denote by $P_n$ and $Q_n$ the corresponding empirical measures. For a fixed realization of the samples and hence of the empirical measures, we can again consider the primal and dual problems~\eqref{QOT} and~\eqref{DQOT}, now with the empirical measures as marginals instead of $(P,Q)$. The primal problem again has a unique solution, denoted $\pi_n\in\Pi(P_n,Q_n)$. Again, there exists a pair $(f_n,g_n)$ solving the dual and describing the primal optimizer via the analogue of~\eqref{eq:primal-dual} with $(P_n,Q_n)$ instead of $(P,Q)$. One difference is that, because the marginal supports are now discrete and hence not connected, $(f_n,g_n)$ are in general non-unique. Another difference is that $(f_n,g_n)$ are a priori only defined on those discrete supports. We may, however, extend them continuously to the population supports $\mathcal{X}$ and $\mathcal{Y}$ (or even all of $\R^d$), in such a way that the extension solves the first-order condition on those sets. Conversely, any such solution defines valid potentials. Passing again to transformed potentials $(\varphi_n,\psi_n)$, the latter satisfy convexity and Lipschitz properties that will be useful below. %

\begin{lemma}[Empirical potentials]\label{le:empirical.potentials}
    Given a realization of the empirical marginals $(P_n,Q_n)$, there exists a pair $(f_n,g_n)\in\mathcal{B}_\oplus$, called empirical potentials, solving the dual problem~\eqref{DQOT} for~$(P_n,Q_n)$ and such that 
    $(\varphi_n,\psi_n)\in\mathcal{B}_\oplus$ defined by 
    \begin{align}\label{eq:convexEmpPotentials}
        \varphi_n(x)=\frac{\|x\|^2}{2}-f_n(x), \qquad \psi_n(y)=\frac{\|y\|^2}{2}-g_n(y)
    \end{align}
    satisfy
    \begin{equation}\label{empirical_optimality}
    \begin{cases}
      \eps=  \int  \left( \langle x, y \rangle - \varphi_n(x)- \psi_n(y) \right)_+ dQ_n(y) & \text{ for all } x\in \mathcal{X},\\  
      \eps=  \int  \left( \langle x, y \rangle - \varphi_n(x)- \psi_n(y)\right)_+ dP_n(x) &\text{ for all } y\in \mathcal{Y}.
    \end{cases}
    \end{equation}    
    Moreover, $\varphi_n: \mathcal{X}\to\R$ is convex and $\mathcal{L}_d$-a.e.\ differentiable with 
    \begin{align}\label{eq:emp-gradient-formula}
    \nabla  \varphi_n(x)= \frac{\int_{\hat{\mathcal{S}}_x} y dQ_n(y)}{Q_n(\hat{\mathcal{S}}_x)} \quad\mbox{for $\mathcal{L}_d$-a.e.~$x\in \mathcal{X}$},  
    \end{align}  
    where $\hat{\mathcal{S}}_x= \{y\in \mathcal{Y}: \varphi_n(x)+ \psi_n(y)\leq \langle x, y \rangle\}$. In particular, $\varphi_n$ is Lipschitz with constant $\|\mathcal{Y}\|_\infty$. An analogous result holds for $\psi_n$.
\end{lemma}

In contrast to \cref{co:potentialGradientsAreLipschitz} for the population potentials, the regularity stated above cannot be improved for the empirical potentials: as $x\mapsto Q_n(\hat{\mathcal{S}}_x)$ is integer-valued, the gradient~\eqref{eq:emp-gradient-formula} is discontinuous and fails to exist at certain points (except in trivial cases).

The empirical potentials $(f_n,g_n)$, or equivalently $(\varphi_n, \psi_n)$, are random and possibly non-unique. For the remainder of the paper, we choose and fix one pair; all subsequent statements are valid as long as the selection is measurable and satisfies~\eqref{empirical_optimality}.

The last result of this section is the consistency of the empirical potentials towards the population counterpart.

\begin{lemma}[Consistency]\label{lemma:consistency}
    We have
    $  \lim_{n\to\infty}\| ( \varphi_n, \psi_n) - (\varphi_\eps, \psi_\eps)\|_\oplus =0  $ a.s.
\end{lemma}

\section{The central limit theorems}\label{se:TheCLTs}

This section states our three central limit theorems for the dual potentials, optimal costs and couplings. The theorem for the potentials is stated in the Banach space $\mathcal{B}_{\oplus}$ defined in~\eqref{eq:defBanach} and the limit is described using the linear operator $[\mathbb{L}]_\oplus$, defined as the composition of  
\begin{align}
    \begin{split}
        \label{Definitio-L}
        \mathbb{L}: \mathcal{B}_\oplus&\to \mathcal{C}(\mathcal{X})\times \mathcal{C}(\mathcal{Y})\\
       \left( \begin{array}{c}
  f\\ g
\end{array} \right) &\mapsto \left( \begin{array}{c}
  f\\ g
\end{array} \right) +  \left( \begin{array}{c}
   \frac{1}{Q(\mathcal{S}_{(\cdot)})}\int_{\mathcal{S}_{(\cdot)}}  g(y) dQ(y)\\
 \frac{1}{P(\mathcal{T}_{(\cdot)})}\int_{\mathcal{T}_{(\cdot)}}  f(x) dP(x)
\end{array} \right)
    \end{split}
\end{align}
and the quotient map $[\cdot ]_\oplus: \mathcal{C}(\mathcal{X})\times \mathcal{C}(\mathcal{Y})\to \mathcal{B}_\oplus$. We write $\mathcal{S}_{(\cdot)}$  and $\mathcal{T}_{(\cdot)}$ for the set-valued mappings $x\mapsto \mathcal{S}_{x}$ and $y\mapsto \mathcal{T}_{y}$ defined in~\eqref{eq:sections}; moreover, we use the shorthand
$$\xi_\eps(x,y):=f_\eps(x)+g_\eps(y)- \frac{1}{2} \|x-y\|^2.$$

\begin{theorem}[CLT for potentials]\label{th:CLT.potentials}
The operator $[\mathbb{L}]_\oplus: \mathcal{B}_\oplus\to\mathcal{B}_\oplus$ admits a bounded inverse $[\mathbb{L}]_\oplus^{-1}$ and we have
$$  \sqrt{n}\left(\begin{array}{c}
    f_n- f_\eps\\
       g_n-g_\eps
\end{array}\right)    \overset{\mathcal{B}_\oplus}{ \rightsquigarrow} -[\mathbb{L}]_\oplus^{-1}\left[  \left( \begin{array}{c}
    \frac{\mathbf{G}_Q}{Q(\mathcal{S}_{(\cdot)})}\\  
    \frac{\mathbf{G}_P}{P(\mathcal{T}_{(\cdot)})} 
\end{array} \right)\right]_\oplus, $$
    where $ ({\bf G}_Q, {\bf G}_P )\in \mathcal{C}(\mathcal{X})\times  \mathcal{C}(\mathcal{Y})$ is the unique-in-law pair of  independent,  tight, centered Gaussian processes with covariances 
    \begin{align*}
    \EE[{\bf G}_Q(x) {\bf G}_Q(x')]&= \EE\big[ (\xi_\eps(x,Y)-\EE[\xi_\eps(x,Y)])(\xi_\eps(x',Y)-\EE[\xi_\eps(x',Y)])\big],\\
    \EE[{\bf G}_P(y) {\bf G}_P(y')]&= \EE\big[ (\xi_\eps(X,y)-\EE[\xi_\eps(X,y)])(\xi_\eps(X,y')-\EE[\xi_\eps(X,y')])\big]
    \end{align*}
    for $X\sim P$ and $Y\sim Q$.
\end{theorem}

The proof of \cref{th:CLT.potentials} is given in \cref{se:proof.CLT.potentials}. Next, we state the central limit theorem for the optimal costs.

\begin{theorem}[CLT for costs]\label{th:CLT.cost}
   We have 
    $$ \sqrt{n}\big( {\rm QOT}(P_n,Q_n) -{\rm QOT}(P,Q) \big)\xrightarrow{w} N(0,\sigma^2), $$
    where the variance $\sigma^2$ is that of the random variable
    \begin{align}\label{eq:variance.CLT.cost}
	f_\eps(X) + g_\eps(Y) - \frac{1}{2\eps}\bigg(\int \left(\xi_\eps(x,Y)\right)_+^2 dP(x)
	+ \int \left(\xi_\eps(X,y) \right)_+^2 dQ(y) \bigg)
    \end{align}
    for $(X,Y)\sim P\otimes Q$.
\end{theorem}

The proof is given in \cref{Section:CLT-plans-and-cost}. Finally, we state the central limit theorem for the optimal couplings; the proof is also reported in \cref{Section:CLT-plans-and-cost}. To describe the limiting variance in a more compact form,  we use the notation $\oplus(f,g):=f\oplus g$.

\begin{theorem}[CLT for couplings]\label{th:CLT.plans}
    For any bounded measurable function $\eta: \mathcal{X}\times\mathcal{Y}\to\mathbb{R}$, 
    $$ \sqrt{n}\left(\int \eta d(\pi_n - \pi_\eps) \right) \xrightarrow{w}  N\left(0, \frac{\sigma^2(\eta)}{\eps^2}\right) $$
    with $\sigma^2(\eta)=\Var(V_X + V_Y)$ where, for $Z\in\{X,Y\}$,
    \begin{align*}
      V_Z :=\mathbb{E}\Bigg[  \int  U(x,y,X,Y) \bar{\eta}(x,y) (\xi_\eps(x,y))_+ d(P\otimes Q)(x,y)
        -\bar{\eta}(X,Y)(\xi_\eps(X,Y))_+  \Bigg\vert Z\Bigg]
    \end{align*}
    for $(X,Y)\sim P\otimes Q$ and $ \bar{\eta}:= \eta - \int_{\xi_\eps\geq 0} \eta  d(P\otimes Q)$ and
    $$
     U(\cdot,\cdot,X,Y) := \oplus\left( [\mathbb{L}]_\oplus^{-1}\left[\left( \begin{array}{c}
           \frac{  (\xi_\eps( \cdot ,Y))_+  } {Q( \mathcal{S}_{(\cdot)})}   \\
            \frac{ (\xi_\eps(X, \cdot ))_+ }{P(\mathcal{T}_{(\cdot)})}  
        \end{array}\right)\right]_{\oplus} \right) .
    $$
\end{theorem}

\subsection{Proof methodology} \label{se:proofMethodology}

The key result is \cref{th:CLT.potentials} on the potentials; \cref{th:CLT.cost,th:CLT.plans} on the costs and couplings are derived from it. We follow the scheme of Z-estimation (see \cite[Chapter~3.3]{vanderVaart.1996}), as in~\cite{gonzalezsanz.2025.sparseregularizedoptimaltransport}, based on the fact that the population and empirical potentials are solutions to certain equations. Specifically, the first-order conditions~\eqref{population_optimality} and~\eqref{empirical_optimality}, which we represent with nonlinear operators $\Phi$ and $\Phi_n$ on $\Banach$ as
$$
\left( \begin{array}{c}
   \eps\\
   \eps
\end{array} \right)=  \Phi  \left( \begin{array}{c}
   \varphi_\eps\\
  \psi_\eps
\end{array} \right), \qquad
\left( \begin{array}{c}
   \eps\\
   \eps
\end{array} \right)=  \Phi_n \left( \begin{array}{c}
   \varphi_n\\
  \psi_n
\end{array} \right).
$$
Following the Z-estimation methodology, we need to establish:
\begin{enumerate}
    \item $ \Phi_n(\varphi_\eps,
   \psi_\eps)- \Phi(\varphi_\eps,
   \psi_\eps)$ satisfies a central limit theorem in $\mathcal{B}_\oplus$, 
\item $\Phi$  is Fréchet differentiable at  $(\varphi_\eps,
   \psi_\eps)$ with invertible derivative,
   \item $\|(\varphi_n,
   \psi_n)-(\varphi_\eps,
   \psi_\eps)\|_{\oplus}\to0$ in probability, 
   \item the following expansion holds,
   \begin{multline*}
      \left\| \Phi_n  \left( \begin{array}{c}
   \varphi_\eps\\
   \psi_\eps
\end{array} \right)- \Phi  \left( \begin{array}{c}
   \varphi_\eps\\
   \psi_\eps
\end{array} \right)- \left( \Phi_n  \left( \begin{array}{c}
   \varphi_n\\
   \psi_n
\end{array} \right)- \Phi  \left( \begin{array}{c}
   \varphi_n\\
   \psi_n
\end{array} \right) \right) \right\|_\oplus \\ = o_{\mathbb{P}} \left( \left\|\left( \begin{array}{c}
   \varphi_n-\varphi_\eps\\
   \psi_n-\psi_\eps
\end{array} \right)  \right\|_{\oplus}+{n^{-\frac{1}{2}}}\right).  
   \end{multline*}
\end{enumerate}
Item (i) follows easily from the fact that 
\begin{multline*}
     \Phi_n  \left( \begin{array}{c}
   \varphi_\eps\\
   \psi_\eps
\end{array} \right)- \Phi  \left( \begin{array}{c}
   \varphi_\eps\\
   \psi_\eps
\end{array} \right)\\ =  \left( \begin{array}{c}
 \frac{1}{n} \sum_{i=1}^n\left( \langle \cdot, Y_i\rangle-  \varphi_\eps(\cdot)-\psi_\eps(Y_i)  \right)_+ -    \int  \left( \langle \cdot, y\rangle-    \varphi_\eps(\cdot)-\psi_\eps(y)  \right)_+ dQ(y) \\[.3em]  
   \frac{1}{n} \sum_{i=1}^n\left( \langle  X_i, \cdot \rangle-    \varphi_\eps(X_i)-\psi_\eps(\cdot)  \right)_+-   \int  \left( \langle x, \cdot \rangle-    \varphi_\eps(x)- \psi_\eps(\cdot )  \right)_+ dP(x)
\end{array} \right)
\end{multline*}
is an average of an i.i.d.\ sequence of uniformly Lipschitz random  functions.
For (ii), the Fréchet differentiability (shown in  \cref{lemma:frechet}) hinges on a novel Lipschitz regularity property of the sections $\mathcal{S}_x$ of the optimal support, formalized in \cref{prop:bound-Lipschitz}. More specifically, recognizing that $\mathcal{S}_x$ is the zero-sublevel set of a convex function, we consider the $\beta$-sublevel sets which correspond to a thickening or shrinking of $\mathcal{S}_x$, and establish a Lipschitz estimate w.r.t.\ $\beta$. This property is one of the key technical innovations of the paper. The invertibility of the derivative is shown in \cref{lemma:invert} via the Fredholm alternative, completing~(ii). The consistency (iii) of the potentials was already stated in  \cref{lemma:consistency}; it follows by a standard Arzelà--Ascoli argument. The main task is to establish~(iv). This property follows from two estimates that we state in \cref{lemma:cross-term} and \cref{lemma:derivative-estimates-empirical-to-pop}, respectively:
\begin{multline*}
         \left\| \Phi_n  \left( \begin{array}{c}
   \varphi_\eps\\
   \psi_\eps
\end{array} \right)- \Phi  \left( \begin{array}{c}
   \varphi_\eps\\
   \psi_\eps
\end{array} \right)- \left( \Phi_n  \left( \begin{array}{c}
   \varphi_n\\
   \psi_n
\end{array} \right)- \Phi  \left( \begin{array}{c}
   \varphi_n\\
   \psi_n
\end{array} \right) \right) \right\|_\oplus \\=o_{\mathbb{P}}\left( \| \nabla  \varphi_n -\nabla \varphi_\eps \|_{\infty} +\| \nabla  \psi_n -\nabla \psi_\eps \|_{\infty}  + n^{-\frac{1}{2}} \right)
     \end{multline*}
and 
$$  \| \nabla  \varphi_n -\nabla \varphi_\eps \|_{\infty}+ \| \nabla  \psi_n -\nabla \psi_\eps \|_{\infty} =\mathcal{O}_{\mathbb{P}} \left(\| ( \varphi_n, \psi_n) - (\varphi_\eps, \psi_\eps)\|_{\oplus} + n^{-\frac{1}{2}}\right).$$
In addition to the aforementioned Lipschitz regularity property of the sections of the optimal support (\cref{prop:bound-Lipschitz}), both estimates rely on another key innovation of this paper, namely to utilize VC theory to bound the complexity of $\mathcal{S}_x$ (and its generalizations by thickening); this gives rise to the term $\mathcal{O}_{\mathbb{P}}(n^{-\frac{1}{2}})$. As similar VC arguments are used several times throughout the proof of the central limit theorems, we formulate a slightly more general estimate as an abstract tool (\cref{pr:VCtool}), with the benefit of streamlining the later proofs. To the best of our knowledge, VC theory has not been used previously in the sample complexity analysis of optimal transport or regularized optimal transport.

\section{Regularity estimates}\label{se:RegularityEstimates}

This section provides several estimates which are key to the proof of the central limit theorems, but also of independent interest.

\subsection{Regularity of the population support and potentials}

The first result studies the section $\mathcal{S}_x$ of the optimal support; it can be seen as the zero-sublevel set of a function involving the population potentials $\varphi_\eps, \psi_\eps$ of \cref{le:population.potentials}. The proposition asserts a Lipschitz regularity with respect to the level~$\beta$ which acts as a parameter to define a thickening (or shrinking) of $\mathcal{S}_x$.

\begin{proposition}\label{prop:bound-Lipschitz}
  Define the set-valued mapping 
   \begin{align}\label{eq:SxBetaDefn}
        \R\ni   \beta\mapsto \mathcal{S}_x(\beta):= \{y\in \mathcal{Y}: \varphi_\eps(x)+ \psi_\eps(y)\leq \langle x, y \rangle +\beta\}.
   \end{align}
There exist $ L,\beta_0>0$  such that 
  $$  \sup_{\|g\|_\infty\leq 1} \sup_{x\in\mathcal{X}}\left|{\int_{\mathcal{S}_x(\alpha)} g(y) dQ(y)}-{\int_{\mathcal{S}_x(\beta)} g(y) dQ(y)} \right|\leq L |\alpha-\beta| \quad\mbox{for all }\alpha,\beta\in (-\beta_0,\beta_0),$$
where the outer supremum is taken over measurable functions $g: \mathcal{Y}\to\R$.
\end{proposition}

Before proving the proposition, we use it to show $\mathcal{C}^{1,1}$-regularity of the population potentials. While this will not be the only application of \cref{prop:bound-Lipschitz}, it is a good illustration of how this abstract result can be used.

\begin{corollary}\label{co:potentialGradientsAreLipschitz}
    The population potential $\varphi_\eps$ belongs to $\mathcal{C}^{1,1}(\mathcal{X})$; that is, its gradient
    \begin{align}\label{eq:potentialGradientsLipschitzExpr}
    x\mapsto \nabla \varphi_\eps(x)= \frac{\int_{\mathcal{S}_x} y dQ(y)}{Q(\mathcal{S}_x)}
    \end{align}
    is Lipschitz continuous. Similarly, the mass $x\mapsto Q(\mathcal{S}_x)$ is Lipschitz continuous. The analogue holds for $\psi_\eps$ and $y\mapsto P(\mathcal{T}_y)$.
\end{corollary}

\begin{proof}
Recall that $\varphi_\eps$ is $\ell$-Lipschitz (with $\ell= \|\mathcal{Y}\|_\infty$) and note that this implies 
$$ \mathcal{S}_{x}(-\ell\|x-z\| )\subset \mathcal{S}_{z}  \subset  \mathcal{S}_{x}(\ell\|x-z\| ) \quad\mbox{for all }x,z\in \mathcal{X}.
$$  
Let $L,\beta_0>0$ be as in \cref{prop:bound-Lipschitz}. Then \cref{prop:bound-Lipschitz} with $\alpha=0$, $\beta=\mp\ell\|x-z\|$ and $g\equiv1$ implies for all $x,z$ with $\ell\|x-z\|\leq \beta_0$ that 
$$
  Q(\mathcal{S}_{x}) - L\ell\|x-z\| 
  \leq Q(\mathcal{S}_{x}(- \ell\|x-z\|))
  \leq Q(\mathcal{S}_{z})
  \leq Q(\mathcal{S}_{x}(\ell\|x-z\|))
  \leq Q(\mathcal{S}_{x}) + L\ell\|x-z\|.
$$
That is, $x\mapsto Q(\mathcal{S}_{x})$ is Lipschitz with constant $L\ell$. Similarly, \cref{prop:bound-Lipschitz} yields that 
$$ x\mapsto \int_{\mathcal{S}_x} y_i dQ(y)
$$
is Lipschitz with constant $L\ell\|\mathcal{Y}\|_\infty$, where $y_i$ denotes the $i$-th component of $y\in\mathcal{Y}$. Recalling the uniform lower bound~\eqref{eq:sectionsBounds} for the denominator in~\eqref{eq:potentialGradientsLipschitzExpr}, it follows that $\nabla \varphi_\eps$ is Lipschitz, which was the claim.
\end{proof}

\begin{remark}
   \Cref{co:potentialGradientsAreLipschitz} links to the regularity theory for unregularized optimal transport, where the gradient of the potential is the optimal transport map. The standard Caffarelli--Urbas conditions \cite{Caffarelli.1996.AnnalsOfMath,Urbas.1997} for Lipschitz continuity of the transport map require in particular that the marginal densities be (bounded and) bounded away from zero and the supports be convex with Lipschitz boundaries. The condition that densities be bounded away from zero is essential, as illustrated by Wang's counterexamples \cite{Wang-PrAmMaSo-95}. 

   In \cref{co:potentialGradientsAreLipschitz}, the regularization parameter $\eps$ is positive and fixed. In light of the aforementioned results, one cannot expect the constant~$L$ to be uniform in~$\eps$, as the estimate would then extend to the unregularized limit. Obtaining such a uniform estimate (under stronger conditions) is an interesting open problem. It is equivalent to a uniform bound for the Hessian of $\varphi_\eps$ and closely related to establishing the sharp rate $\eps^{\frac{1}{d+2}}$ for the diameter of~$\mathcal{S}_x$ as $\eps\to 0$ (which was shown in~\cite{GonzalezSanzNutz.24b} for the case $d=1$).
\end{remark}

\begin{proof}[Proof of \cref{prop:bound-Lipschitz}] 
Fix $x_0\in \mathcal{X}$. Recall from \cref{le:population.potentials} that $\mathcal{S}_{x_0}$ is convex, that ${\rm int}\, \mathcal{S}_{x_0}\neq\emptyset$, and the representation~\eqref{eq:sectionsBoundaries}. Changing the coordinate system, we can assume without loss of generality that 
$$
  0\in  {\rm int}\, \mathcal{S}_{x_0} =\{y\in\mathcal{Y}:\, \varphi_\eps(x_0)+\psi_\eps(y)-\langle x_0,y \rangle < 0\}.
$$
As $\varphi_\eps$ and $\psi_\eps$ are Lipschitz, this representation yields that there exists $\delta_0>0$ such that $0\in {\rm int}\, \mathcal{S}_{x}(\delta )$ for all $(x,\delta)\in\mathcal{X}\times \R$ with $|\delta|+\|x-x_0\|\leq \delta_0$.

Consider such $(x,\delta)$ and let $g:\R^d\to \R$ be bounded and measurable with $\|g\|_\infty\leq 1$. By the co-area formula, 
$$  \int_{\mathcal{S}_{x}(\delta )} g(y) dQ(y) = \int_{\mathbb{S}^{d-1}}  
\int_{0}^{t(x,v,\delta)}  g(r v) q(rv)  dr d\mathcal{H}^{d-1}(v),  $$
where $t(x,v,\delta)$ is the unique number  $ t>0 $ such that 
$ t v \in  
\partial \mathcal{S}_{x}(\delta )$, $q$ denotes the Lebesgue density of~$Q$ (which is bounded by \cref{Assumption:reg}) and $\mathcal{H}^{d-1}$ denotes $(d-1)$-dimensional Hausdorff measure on~ $\mathbb{S}^{d-1}$. Next, we show that  $t(x,v,\delta)$ is locally Lipschitz. By \cite[Theorem~5.4]{DelfourZolesio}, there exists a convex Lipschitz function $h_Q$ such that 
$\mathcal{Y}= \{ y: h_Q(y)\leq 0\}$ and $\partial \mathcal{Y}= \{ y: h_Q(y)= 0\}$. We then introduce the function
$$\gamma(t, x,v, \delta):=\max\left\{ \varphi_\eps(x)+\psi_\eps( t v ) - t \langle v, x\rangle +\delta, h_Q(t v) \right\}.$$
Note that 
$t\mapsto \gamma(t, x,v, \delta)$ is convex, that $0\in {\rm int}\, \mathcal{S}_{x}(\delta )$ translates to $\gamma(0,x,v,\delta )<0$, and 
$$  \gamma(t(x,v,\delta),x,v,\delta )=0. $$
We can see this equation as an implicit definition of $t(x,v,\delta)$. Indeed, fix $v_0\in \mathbb{S}^{d-1}$ and note that $\gamma(t(x_0,v_0,0),x_0,v_0,0)=0$ and $\gamma (0,  x_0,v_0,0)<0$ imply that $t(x_0,v_0,0)$ cannot be a minimum of $t\mapsto \gamma (t,x_0,v_0,0)$. As a consequence, zero is not in the subdifferential,
\begin{multline*}
    0\not\in \partial_1 \gamma (t(x_0,v_0,0),  x_0,v_0,0)\\=\{ t\in \R: \,  \gamma (t(x_0,v_0,0),  x_0,v_0,0) \leq  \gamma (s,  x_0,v_0,0) + t\cdot (t(x_0,v_0,0)-s),  \ \forall \, s\in \R\}.
\end{multline*}
 As $\gamma(t, x,v, \delta)$ is jointly Lipschitz in a neighborhood of $(t(x_0,v_0,0),x_0,v_0,0)$, 
Clarke's implicit function theorem (more precisely, the corollary on \cite[p.~256]{Clarke}) then yields a neighborhood of $(x_0,v_0,0)$ where the function  $t(x,v,\delta)$ is Lipschitz with some constant $\ell$. By a compactness argument, the constant $\ell$ can be taken uniformly in the variable $v\in \mathbb{S}^{d-1}$. As a consequence, there exists $\delta_0$ such that  $|t(x,v,\delta)-t(x,v,\beta)|\leq \ell  |\delta-\beta|$ for all $x,\delta,\beta$ with $\|x-x_0\|<\delta_0$ and $\delta,\beta\in (-\delta_0, \delta_0)$ and all $v\in \mathbb{S}^{d-1}$.
Since $g$ and $q$ are bounded, it follows for all such $x,\delta,\beta$ that  
\begin{align}\label{eq:bound-Lipschitz-proof}
\left|\int_{\mathcal{S}_{x}(\delta )} g(y) dQ(y)- \int_{\mathcal{S}_{x}(\beta )} g(y) dQ(y)\right| \leq \ell \| g\|_\infty \| q\|_\infty |\beta-\delta|.
\end{align}
Here $\ell$ and $\delta_0$ depend on~$x_0$, and~\eqref{eq:bound-Lipschitz-proof} holds only for~$x$ in a ball around $x_0$. However, the compact set $\mathcal{X}$ is covered by finitely many such balls, hence~\eqref{eq:bound-Lipschitz-proof} holds uniformly in $x\in\mathcal{X}$ after possibly changing~$\ell$ and~$\delta_0$.
\end{proof}

\subsection{VC bounds for statistical complexity of the optimal support}

The next proposition encapsulates a tool that will be used in several of the most important steps towards the central limit theorems. It uses VC theory to bound the statistical complexity of the sublevel sets $\mathcal{S}_x(\cdot)$ from~\eqref{eq:SxBetaDefn} to guarantee that a given class of functions satisfies a Donsker-type property.
        
\begin{proposition}\label{pr:VCtool}
    Let $g:\mathcal{Y}\to\R$ be bounded measurable and let $\Lambda$ be any index set. For each $\lambda\in\Lambda$ and $x\in\mathcal{X}$, consider two measurable random functions  $V_{n,\lambda,x}, W_{n,\lambda,x}: \mathcal{Y}\to \R$ such that 
    $$\1_{\mathcal{S}_x(\alpha_n)} \leq V_{n,\lambda,x},W_{n,\lambda,x} \leq \1_{\mathcal{S}_x(\beta_n)} \quad\mbox{$Q$-a.s.\ for all~$n\in\NN$}$$
    for some real-valued random variables $\alpha_n,\beta_n$ with $\alpha_n\to0$ and $\beta_n\to0$ in probability~$\mathbb{P}$.
    Then 
    \begin{align*}
        \sup_{\lambda\in\Lambda,\,x\in\mathcal{X}}\left|\int V_{n,\lambda,x}(y)g(y) dQ(y)-\int W_{n,\lambda,x}(y)g(y) dQ_n(y)\right| = \mathcal{O}_{\mathbb{P}}\left(n^{-\frac{1}{2}} + |\alpha_n-\beta_n|\right).
    \end{align*}
\end{proposition}

\begin{proof}
Decomposing $g=(g)_+-(g)_+$ and using the triangle inequality, we may assume w.l.o.g.\ that $g\geq0$. Set
\begin{align*}
D_{n,\lambda,x}:= \int V_{n,\lambda,x}(y)g(y) dQ(y)-\int W_{n,\lambda,x}(y)g(y) dQ_n(y).
\end{align*}
As $g\geq0$, the assumed inequalities yield
 $$
  D_{n,\lambda,x} \leq \int_{\mathcal{S}_x(\beta_n)} g(y)dQ(y) - \int_{\mathcal{S}_x(\alpha_n)} g(y)dQ_n(y)
 $$
 and hence
 \begin{equation}\label{eq:DupperboundGeneral}
     D_{n,\lambda,x}
     \leq  \left|\int_{\mathcal{S}_x(\beta_n)} g(y)dQ(y)- \int_{\mathcal{S}_x(\alpha_n)} g(y)dQ(y)  \right| 
     +\left|\int_{\mathcal{S}_x(\alpha_n)} g(y)d(Q-Q_n)(y) \right|.
 \end{equation}
Let $L,\beta_0>0$ be the constants from \cref{prop:bound-Lipschitz}. Then by \cref{prop:bound-Lipschitz}, the first term on the right-hand side of~\eqref{eq:DupperboundGeneral} is bounded by 
$ \|g\|_\infty L |\alpha_n-\beta_n|$ on the event $\{|\alpha_n|,|\beta_n|\leq \beta_0\}$ which satisfies $\mathbb{P}(|\alpha_n|,|\beta_n|\leq \beta_0)\to1$ as $n\to\infty$.

We can thus focus on the last term of~\eqref{eq:DupperboundGeneral}; our goal is to show that it is of order $\mathcal{O}_{\mathbb{P}} (n^{-\frac{1}{2}})$. Note that
\begin{equation*}
     \left|\int_{\mathcal{S}_x(\alpha_n)} g(y)d(Q-Q_n)(y) \right| \leq \sup_{\delta\in\R}\left|\int_{\mathcal{S}_x(\delta)} g(y)d(Q-Q_n)(y) \right|.
\end{equation*}
We use  VC theory (see \cite[Section~2.6]{vanderVaart.1996}) to bound the right-hand side. 
Define the VC dimension (or index) of a collection $\mathcal{C}$ of subsets of a set~$\mathcal{X}$ as in \cite[Section~2.6.1]{vanderVaart.1996}, namely
$$ \mathcal{V}(\mathcal{C}):= \inf \left\{ n\in \NN : \max_{y_1, \dots, y_n\in \mathcal{X}}\Delta_n(\mathcal{C}, y_1, \dots, y_n) <2^n \right\}$$
where 
$$ \Delta_n(\mathcal{C}, y_1, \dots, y_n) := \#\left\{ C \cap \{ y_1, \dots, y_n\}:   C\in \mathcal{C}  \right\}. $$
Recall also that the VC dimension of a class $\mathcal{F}$ of functions $f:\mathcal{X}\to\R$  is defined as the VC dimension of the collection of all its subgraphs $\{{(y,t)\in\mathcal{X}\times \R : t< f(y) }\}.$

Consider the collection 
$$ \mathcal{C}:=\{ \mathcal{S}_x(\delta): \ \delta\in\R  , \ x\in \mathcal{X} \}$$
of subsets of $\mathcal{Y}$. To verify that $\mathcal{V}(\mathcal{C})<\infty$, we consider the following linear space of functions on~$\mathcal{Y}$,
$$ \mathcal{G}:=\left\{y\mapsto  \beta_1 \psi(y)-\langle x,y \rangle + \beta_2: (x, \beta_1, \beta_2)\in \R^{d+2}\right\}$$
and note that $\mathcal{G}$ has algebraic dimension $d+2$. By \cite[Lemma~2.6.15]{vanderVaart.1996}, it follows that 
$ \mathcal{V}(\mathcal{G}) \leq  d+4<\infty$. Given a function class of finite VC dimension, the class of its zero-sublevel sets has finite VC dimension by \cite[Lemmas~2.6.17 and~2.6.18]{vanderVaart.1996}. Noting 
$$
  \mathcal{C} \subset \{g\in \mathcal{G}: g\leq 0\},
$$
it follows that $ \mathcal{V}(\mathcal{C})<\infty$. Using \cite[Lemma~2.6.18]{vanderVaart.1996}, this implies that the function class
$$ \mathcal{F} := \{g \1_{C} : \ C\in \mathcal{C}\}$$ has finite VC dimension. Hence, \cite[Theorem~2.6.8]{vanderVaart.1996} yields that $\mathcal{F}$ is $Q$-Donsker. (The condition of \cite[Theorem~2.6.8]{vanderVaart.1996} holds due to the uniform boundedness of the functions in~$\mathcal{F}$, as observed in the proof of that theorem in~\cite{vanderVaart.1996}.) We conclude by \cite[Corollary~2.3.12]{vanderVaart.1996} that 
$$ \sup_{x\in \mathcal{X}, \, \delta\in \R} \left|\int_{\mathcal{S}_x(\delta)} g(y)d(Q-Q_n)(y) \right| = \mathcal{O}_{\mathbb{P}}\left(n^{-\frac{1}{2}}\right).
$$
So far, we have shown that $\sup_{\lambda,x}D_{n,\lambda,x} \leq \mathcal{O}_{\mathbb{P}}(n^{-\frac{1}{2}} + |\alpha_n-\beta_n|)$. Analogously, we obtain a lower bound of the same order by exchanging the roles of $\alpha_n$ and $\beta_n$ in the above argument, completing the proof.
\end{proof}

Next, we record a version of \cref{pr:VCtool} on the product space $\mathcal{X}\times\mathcal{Y}$ instead of~$\mathcal{Y}$; it will be used in the  proof of the central limit theorem for the optimal couplings (\cref{th:CLT.plans}). The proof is analogous to the one of  \cref{pr:VCtool} and hence omitted.

\begin{proposition}\label{pr:VCtoolProduct}
   Define 
   \begin{align}\label{eq:MBetaDefn}
        \mathcal{M}(\beta):= \{(x,y)\in \mathcal{X}\times\mathcal{Y}: \varphi_\eps(x)+ \psi_\eps(y)\leq \langle x, y \rangle +\beta\} \quad \mbox{for }\beta\in\R.
   \end{align}
    Let $g:\mathcal{X}\times\mathcal{Y}\to\R$ be bounded measurable and let $\Lambda$ be any index set. For each $\lambda\in\Lambda$, consider two measurable random functions  $V_{n,\lambda}, W_{n,\lambda}: \mathcal{X}\times\mathcal{Y}\to \R$ such that 
    $$\1_{\mathcal{M}(\alpha_n)} \leq V_{n,\lambda},W_{n,\lambda} \leq \1_{\mathcal{M}_x(\beta_n)} \quad\mbox{$P\otimes Q$-a.s.\ for all~$n\in\NN$}$$
    for some real-valued random variables $\alpha_n,\beta_n$ with $\alpha_n\to0$ and $\beta_n\to0$ in probability~$\mathbb{P}$.
    Then 
    \begin{multline*}
        \sup_{\lambda\in\Lambda}\left|\int V_{n,\lambda}(x,y)g(x,y) d(P\otimes Q)(x,y)-\int W_{n,\lambda}(x,y)g(x,y) d(P_n\otimes Q_n)(x,y)\right| \\= \mathcal{O}_{\mathbb{P}}\left(n^{-\frac{1}{2}} + |\alpha_n-\beta_n|\right).
    \end{multline*}
\end{proposition}

\subsection{Statistical gradient estimate}

Our first application of \cref{pr:VCtool} is a gradient estimate of statistical nature: we bound the difference $\nabla  \varphi_n -\nabla \varphi_\eps$ of the empirical and population gradients by the difference of the potentials themselves and an error term of order $\mathcal{O}_{\PP}(n^{-\frac{1}{2}})$. This estimate will be instrumental for our subsequent arguments towards the CLT of \cref{th:CLT.potentials}, specifically for \cref{lemma:cross-term}.  

\begin{lemma}\label{lemma:derivative-estimates-empirical-to-pop}
    We have
    $$  \| \nabla  \varphi_n -\nabla \varphi_\eps \|_{\infty} =\mathcal{O}_{\mathbb{P}} \left(n^{-\frac{1}{2}}+ \| ( \varphi_n, \psi_n) - (\varphi_\eps, \psi_\eps)\|_\oplus \right). $$
    As a consequence, the norm $\|\cdot\|_{\oplus,1} $  on $\mathcal{B}^{0,1}_\oplus=(\mathcal{C}^{0,1}(\mathcal{X})\times \mathcal{C}^{0,1}(\mathcal{Y}))/\sim_\oplus$ satisfies
    $$ \| ( \varphi_n, \psi_n) - (\varphi_\eps, \psi_\eps)\|_{\oplus, 1}= \mathcal{O}_{\mathbb{P}} \left(n^{-\frac{1}{2}}+ \| ( \varphi_n, \psi_n) - (\varphi_\eps, \psi_\eps)\|_\oplus \right). $$
\end{lemma}

\begin{proof}
    Recall \cref{eq:population-gradient-formula,eq:emp-gradient-formula}, and in particular the random sets $\hat{\mathcal{S}}_x=\hat{\mathcal{S}}_{x,n}$. Decompose
    \begin{align}\label{eq:gradientDiffDecomp}
     \nabla   \varphi_\eps(x)-\nabla \varphi_n(x)&= \frac{\int_{\mathcal{S}_x} y dQ(y)}{Q(\mathcal{S}_x)}-\frac{\int_{\hat{\mathcal{S}}_x} y dQ_n(y)}{Q_n(\hat{\mathcal{S}}_x)} \nonumber\\
        &=  \underbrace{\frac{\int_{\mathcal{S}_x} y dQ(y)-\int_{\hat{\mathcal{S}}_x} y dQ_n(y)}{Q(\mathcal{S}_x)}}_{A_n(x)}+ \underbrace{\frac{\int_{\hat{\mathcal{S}}_x} y dQ_n(y) }{Q_n(\hat{\mathcal{S}}_x) Q(\mathcal{S}_x)}  \left( Q_n(\hat{\mathcal{S}}_x) -Q(\mathcal{S}_x)\right)}_{B_n(x)}.
    \end{align}
We first focus on $A_n(x)$. In view of the uniform bound~\eqref{eq:sectionsBounds}, it is enough to estimate the real-valued function
\begin{align}\label{eq:VCproofDefDn}
x\mapsto D_n(x):= \int_{\mathcal{S}_x} y_i dQ(y)-\int_{\hat{\mathcal{S}}_x} y_i dQ_n(y),
\end{align}
 where $y_i$ is the $i$-th component of~$y$. Recalling the definition of $\mathcal{S}_x(\cdot)$ from~\eqref{eq:SxBetaDefn}, we check 
 $$ \mathcal{S}_x(-\| ( \varphi_n, \psi_n) - (\varphi_\eps, \psi_\eps)\|_\oplus ) \subset \hat{\mathcal{S}}_x \subset \mathcal{S}_x(\| ( \varphi_n, \psi_n) - (\varphi_\eps, \psi_\eps)\|_\oplus ).$$
Moreover, we trivially have $\mathcal{S}_x(-\| ( \varphi_n, \psi_n) - (\varphi_\eps, \psi_\eps)\|_\oplus ) \subset \mathcal{S}_x\subset\mathcal{S}_x(\| ( \varphi_n, \psi_n) - (\varphi_\eps, \psi_\eps)\|_\oplus )$. We can thus apply \cref{pr:VCtool} with
$$ g(y)=y_i, \quad \Lambda=\emptyset, \quad V_{n,\lambda,x} = \1_{\mathcal{S}_x}, \quad W_{n,\lambda,x} = \1_{\hat{\mathcal{S}}_x}, \quad \beta_n=-\alpha_n=\| ( \varphi_n, \psi_n) - (\varphi_\eps, \psi_\eps)\|_\oplus
$$
to conclude that $\sup_x |A_n(x)| = \mathcal{O}_{\mathbb{P}} (n^{-\frac{1}{2}}+ \| ( \varphi_n, \psi_n) - (\varphi_\eps, \psi_\eps)\|_\oplus )$.

It remains to bound the second term, $B_n(x)$, of~\eqref{eq:gradientDiffDecomp}. Noting the inequality 
$$|B_n(x)|\leq \|\mathcal{Y}\|_\infty \frac{|Q_n(\hat{\mathcal{S}}_x) -Q(\mathcal{S}_x)|}{Q(\mathcal{S}_x)} $$
and recalling from~\eqref{eq:sectionsBounds} the uniform lower bound for the denominator, it suffices to estimate $Q_n(\hat{\mathcal{S}}_x) -Q(\mathcal{S}_x)$. We apply \cref{pr:VCtool} (with the same data as above except that now $g(y)\equiv 1$) to get 
$\sup_x |Q_n(\hat{\mathcal{S}}_x) -Q(\mathcal{S}_x)| = \mathcal{O}_{\mathbb{P}} (n^{-\frac{1}{2}}+ \| ( \varphi_n, \psi_n) - (\varphi_\eps, \psi_\eps)\|_\oplus )$.
\end{proof}

\section{Proof of the CLT for potentials}\label{se:proof.CLT.potentials}

This section is devoted to the proof of \cref{th:CLT.potentials}. 
As outlined in \cref{se:proofMethodology}, we  follow the procedure for Z-estimation problems in empirical process theory. Define the nonlinear operator 
$$ \Banach\ni  \left( \begin{array}{c}
   f\\
   g
\end{array} \right)  \mapsto  \Phi  \left( \begin{array}{c}
   f\\
   g
\end{array} \right)=\left( \begin{array}{c}
    \int  \left( \langle \cdot, y\rangle- f(\cdot)-g(y)  \right)_+ dQ(y) \\  
     \int  \left( \langle x, \cdot \rangle- f(x)- g(\cdot )  \right)_+ dP(x)
\end{array} \right) \in  \mathcal{C}(\mathcal{X})\times \mathcal{C}(\mathcal{Y}) $$
and its empirical version 
$$ \Banach\ni  \left( \begin{array}{c}
   f\\
   g
\end{array} \right)  \mapsto  \Phi_n  \left( \begin{array}{c}
   f\\
   g
\end{array} \right)=\left( \begin{array}{c}
    \int  \left( \langle \cdot, y\rangle- f(\cdot)- g(y)  \right)_+ dQ_n(y) \\  
     \int  \left( \langle x, \cdot \rangle- f(x)- g(\cdot )  \right)_+ dP_n(x)
\end{array} \right) \in  \mathcal{C}(\mathcal{X})\times \mathcal{C}(\mathcal{Y}) .$$
We denote by $[\Phi]_\oplus$ the composition of $\Phi$ with the quotient map $[\cdot]_\oplus$ of ~$\sim_\oplus$, and similarly for $[\Phi_n]_\oplus$. Below, we first show that $\Phi$ is differentiable at the population potentials, with invertible derivative. Then, we show that 
$$ \sqrt{n}\left( \Phi_n \left( \begin{array}{c}
   \varphi_\eps\\
   \psi_\eps
\end{array} \right)- \Phi \left( \begin{array}{c}
   \varphi_\eps\\
   \psi_\eps
\end{array} \right) \right)  $$
converges to a tight Gaussian element of $\mathcal{C}(\mathcal{X})\times \mathcal{C}(\mathcal{Y})$. Finally, we show that the remainder 
$$ [\Phi_n-\Phi] \left( \begin{array}{c}
   \varphi_\eps-\varphi_n\\
   \psi_\eps-\psi_n
\end{array} \right) =o_\PP\left( n^{-\frac{1}{2}}+\left\| \left( \begin{array}{c}
   \varphi_\eps-\varphi_n\\
   \psi_\eps-\psi_n
\end{array} \right)\right\|_{\oplus}\right)  $$
tends to zero in probability. 
\subsection{Differentiability and invertibility of the optimality conditions}

Our first result establishes the Fr\'echet differentiability of 
    $\Phi$ at the point $(\varphi_\eps,\psi_\eps)$. We recall that $\Phi$ is Fr\'echet differentiable at $(f_*,g_*)$ with derivative
    $  {\rm D}\Phi_{(f_*,g_*)}$ if 
    $$ \lim_{\|(f,g)\|_{\oplus} \to 0} \frac{ \left\|\Phi(f_*+f,g_*+g)-\Phi(f_*,g_*) - {\rm D}\Phi_{(f_*,g_*)} (f,g) \right\|_{\mathcal{C}(\mathcal{X})\times 
    \mathcal{C}(\mathcal{Y})} }{\|(f,g)\|_{\oplus}}=0.$$
\begin{lemma}\label{lemma:frechet}
    The function $\Phi: \mathcal{B}_\oplus\to \mathcal{C}(\mathcal{X})\times \mathcal{C}(\mathcal{Y}) $ is Fr\'echet differentiable at $(\varphi_\eps,\psi_\eps)$ with derivative 
    $$  {\rm D}\Phi_{(\varphi_\eps,\psi_\eps)}\left( \begin{array}{c}
    f\\
    g
\end{array} \right) = \left( \begin{array}{c}
   {\rm D}\Phi_1  (f,g)\\
       {\rm D}\Phi_2  (f,g)
\end{array} \right), $$
where 
$$ {\rm D}\Phi_1  (f,g)(x)= - f(x) Q({\mathcal{S}_x}) -\int_{\mathcal{S}_x} g(y) dQ(y) \quad \text{for }x\in \mathcal{X},$$
$$ {\rm D}\Phi_2  (f,g)(y)= - g(y) P({\mathcal{T}_y}) -\int_{\mathcal{T}_y} f(x) dP(x) \quad \text{for }y\in \mathcal{Y} .  $$
\end{lemma}

\begin{proof}
Given $(h_1,h_2)\in\Banach$, set
    \begin{align*}
    A(h_1,h_2)(x)&:=    \int  \left( \langle x, y\rangle-(\varphi_\eps+h_1) (x)-(\psi_\eps+h_2)(y)  \right)_+ dQ(y) \\
    &\qquad - \int  \left( \langle x, y\rangle- \varphi_\eps(x)-\psi_\eps(y)  \right)_+ dQ(y).
    \end{align*}
    With the expression for ${\rm D}\Phi_1 $ stated in the lemma, our goal is to show that 
    \begin{equation}
    \label{eq:claim-frechet0}
    \lim_{\|(h_1,h_2)\|_{\oplus} \to 0}  \frac{\|A(h_1,h_2) - {\rm D}\Phi_1 (h_1,h_2) \|_\infty}{\|(h_1,h_2)\|_{\oplus}} =0.
    \end{equation}
    For any $a,b\in\R$, applying the fundamental theorem of calculus  to the absolutely continuous function $\phi(t)=((1-t)b+ ta)_+$ yields 
    \begin{align}\label{eq:fundamentalTheoremCalc}(a)_+ -  (b)_+ = \phi(1)-\phi(0)=(a-b) \mathcal{L}_1(t\in [0,1]: {(1-t)b  + t a\geq 0 } ).
    \end{align}    
    Set 
    $$\mathcal{J}_{x}=\{ (t,y)\in [0,1]\times \mathcal{Y}  :\, \langle x, y\rangle- \varphi_\eps(x)-\psi_\eps(y) -t h_1 (x)-  th_2(y)\geq 0\}, $$
    then the above yields
    \begin{align*}
    A(h_1,h_2)(x)&=-h_1 (x) (\mathcal{L}_1 \otimes Q)(\mathcal{J}_{x})
        - \int_{\mathcal{J}_{x}} h_2 (y) d(\mathcal{L}_1 \otimes Q)(t,y).
    \end{align*}
Inserting this expression into~\eqref{eq:claim-frechet0}, we see that a sufficient condition for~\eqref{eq:claim-frechet0} is
\begin{equation}
    \label{eq:claim-frechet}
   \sup_{\|g\|_\infty\leq 1} \left\| \int_{\mathcal{J}_{(\cdot)}} g (y) d(\mathcal{L}_1 \otimes Q)(t,y)- \int_{\mathcal{S}_{(\cdot)}} g(y) dQ(y)\right\|_\infty \to 0 \quad\mbox{as $\|(h_1,h_2)\|_{\oplus} \to 0 $.}
\end{equation}
Fix $x\in \mathcal{X}$ and $g\in \mathcal{C}(\mathcal{Y})$ with $\|g\|_\infty\leq 1$. Decomposing $g=(g)_+-(g)_-$ and using the triangle inequality, we may assume without loss of generality that $g\geq 0$. Recall the notation $ \mathcal{S}_x(\beta)= \{y: \varphi_\eps(x)+ \psi_\eps(y)\leq \langle x, y \rangle +\beta\} $ from~\eqref{eq:SxBetaDefn}.
We observe that 
\begin{align*}
    \mathcal{J}_{x}&\subset \{ (t,y)\in [0,1]\times \mathcal{Y}  :\, \langle x, y\rangle- \varphi_\eps(x)-\psi_\eps(y)\geq -\|(h_1,h_2)\|_{\oplus} \}\\
    &=[0,1]\times  \mathcal{S}_x(\|(h_1,h_2)\|_{\oplus})
\end{align*}
and similarly
$ [0,1]\times  \mathcal{S}_x(-\|(h_1,h_2)\|_{\oplus}) \subset \mathcal{J}_{x} $. As $g\geq0$, it follows that
$$ \int_{\mathcal{S}_{x}(-\|(h_1,h_2)\|_{\oplus})} g (y) dQ(y)\leq  \int_{\mathcal{J}_{x}} g (y) d(\mathcal{L}_1 \otimes Q)(t,y) \leq \int_{\mathcal{S}_{x}(\|(h_1,h_2)\|_{\oplus})} g (y) dQ(y).$$
By \cref{prop:bound-Lipschitz}, the upper and lower bounds both converge to $\int_{\mathcal{S}_{x}} g (y) dQ(y)$, uniformly in~$x$ and~$g$, as  $\|(h_1,h_2)\|_{\oplus} \to 0$. This completes the proof of~\eqref{eq:claim-frechet} and hence of the lemma.
\end{proof}

\cref{lemma:frechet} implies that 
$$ \left\| \Phi \left(\begin{array}{c}
     \varphi_\eps\\
     \psi_\eps
\end{array}\right)-\Phi \left(\begin{array}{c}
     \varphi_n\\
     \psi_n
\end{array}\right)- {\rm D} \Phi_{(\varphi_\eps,\psi_\eps)} \left(\begin{array}{c}
     \varphi_\eps -\varphi_n\\
       \psi_\eps-\psi_n
\end{array}\right)\right\|_{\mathcal{C}(\mathcal{X})\times \mathcal{C}(\mathcal{Y})}= o_\PP \left( \left\| \left(\begin{array}{c}
     \varphi_\eps -\varphi_n\\
       \psi_\eps-\psi_n
\end{array}\right) 
\right\|_\oplus \right).$$
Using also the uniform bounds~\eqref{eq:sectionsBounds}, it follows that
\begin{equation}\label{development-Phi-tilde}
 \left\|   \tilde\Phi \left(\begin{array}{c}
     \varphi_\eps\\
     \psi_\eps
\end{array}\right)-\tilde \Phi \left(\begin{array}{c}
     \varphi_n\\
     \psi_n
\end{array}\right)+ \mathbb{L} \left(\begin{array}{c}
     \varphi_\eps -\varphi_n\\
       \psi_\eps-\psi_n
\end{array}\right) \right\|_{\mathcal{C}(\mathcal{X})\times \mathcal{C}(\mathcal{Y})}= o_\PP \left( \left\| \left(\begin{array}{c}
     \varphi_\eps -\varphi_n\\
       \psi_\eps-\psi_n
\end{array}\right) 
\right\|_\oplus \right),
\end{equation}
where 
$$ \tilde\Phi \left(\begin{array}{c}
     f\\
     g
\end{array}\right) := \left( \begin{array}{c}
    \frac{1}{Q(\mathcal{S}_{(\cdot)})}\int  \left( \langle \cdot, y\rangle- f(\cdot)-g(y)  \right)_+ dQ(y) \\  
    \frac{1}{P(\mathcal{T}_{(\cdot)})} \int  \left( \langle x, \cdot \rangle- f(x)- g(\cdot )  \right)_+ dP(x)
\end{array} \right)$$
and $\mathbb{L} := {\rm I}+ \mathbb{A} $
with ${\rm I}$ denoting the identity operator and
$$ \mathbb{A} \left( \begin{array}{c}
    f\\
    g
\end{array} \right) := \left( \begin{array}{c}
     \mathbb{A}_1( g)\\
    \mathbb{A}_2( f)
\end{array} \right) = \left( \begin{array}{c}
   \frac{1}{Q(\mathcal{S}_{(\cdot)})}\int_{\mathcal{S}_{(\cdot)}}  g(y) dQ(y)\\
 \frac{1}{P(\mathcal{T}_{(\cdot)})}\int_{\mathcal{T}_{(\cdot)}}  f(x) dP(x)
\end{array} \right).$$
In the preceding displays, it is tacitly understood that the left-hand side takes two arguments $(x',y')$; on the right-hand side, $x'$ is inserted into $\mathcal{S}_{(\cdot)}$ and $f(\cdot)$, whereas $y'$ is inserted into $\mathcal{T}_{(\cdot)}$ and $g(\cdot)$.

We denote by $[\mathbb{L}]_\oplus$ the composition of $\mathbb{L}$ with the quotient map $[\cdot]_\oplus$. 
\begin{lemma}\label{lemma:invert}
    The operator $[\mathbb{L}]_\oplus: \mathcal{B}_\oplus\to \mathcal{B}_\oplus$ is a bounded bijection. 
\end{lemma}
\begin{proof}
    As $\Banach$ is a Banach space, it suffices to show that the bounded linear operator $[\mathbb{L}]_\oplus$ is invertible. Indeed, the operator ${\mathbb{A}}$ is compact by \cite[Lemma 4.2]{GonzalezSanzNutzRiveros.25gradDesc}. Hence, by the Fredholm alternative, $[\mathbb{L}]_\oplus= [{\rm I}+ \mathbb{A}]_\oplus$ is invertible if and only if $\mathbb{L}(f,g) \sim_\oplus 0 $ implies $ (f,g)\sim_\oplus  0$. In view of the definition of $\mathbb{L}$, the latter is in turn equivalent to  ${\mathbb{A}}(f,g)\sim_\oplus -(f,g) $ implying $ (f,g)\sim_\oplus 0$, and that holds by \cite[Lemma 4.3]{GonzalezSanzNutzRiveros.25gradDesc}.
 \end{proof}

Taking equivalence classes in  \cref{development-Phi-tilde}, we have
$$  \left\|   [\tilde\Phi]_\oplus \left(\begin{array}{c}
     \varphi_\eps\\
     \psi_\eps
\end{array}\right)-[\tilde \Phi]_\oplus \left(\begin{array}{c}
     \varphi_n\\
     \psi_n
\end{array}\right)+ [\mathbb{L}]_\oplus \left(\begin{array}{c}
     \varphi_\eps -\varphi_n\\
       \psi_\eps-\psi_n
\end{array}\right) \right\|_{\oplus}= o_\PP \left( \left\| \left(\begin{array}{c}
     \varphi_\eps -\varphi_n\\
       \psi_\eps-\psi_n
\end{array}\right) 
\right\|_\oplus \right). $$
Setting also
$$ \tilde\Phi_n \left(\begin{array}{c}
     f\\
     g
\end{array}\right) := \left( \begin{array}{c}
    \frac{1}{Q(\mathcal{S}_{(\cdot)})}\int  \left( \langle \cdot, y\rangle-  f(\cdot)- g(y)  \right)_+ dQ_n(y) \\  
    \frac{1}{P(\mathcal{T}_{(\cdot)})} \int  \left( \langle x, \cdot \rangle-  f(x)-  g(\cdot )  \right)_+ dP_n(x)
\end{array} \right),$$
the first-order conditions~\eqref{population_optimality} and~\eqref{empirical_optimality} imply
$$ [\tilde\Phi]_\oplus \left(\begin{array}{c}
     \varphi_\eps\\
     \psi_\eps
\end{array}\right) = \left( \begin{array}{c}
    \frac{\eps}{Q(\mathcal{S}_{(\cdot)})}\\  
    \frac{\eps}{P(\mathcal{T}_{(\cdot)})} 
\end{array} \right)= [\tilde\Phi_n]_\oplus \left(\begin{array}{c}
     \varphi_n\\
     \psi_n
\end{array}\right)$$
and we conclude that 
\begin{equation}
    \label{eq:development-before-inverting-and-cross}
     \left\|   [\tilde\Phi_n]_\oplus \left(\begin{array}{c}
     \varphi_n\\
     \psi_n
\end{array}\right)-[\tilde \Phi]_\oplus \left(\begin{array}{c}
     \varphi_n\\
     \psi_n
\end{array}\right)+ [\mathbb{L}]_\oplus \left(\begin{array}{c}
     \varphi_\eps -\varphi_n\\
       \psi_\eps-\psi_n
\end{array}\right) \right\|_{\oplus}= o_\PP \left( \left\| \left(\begin{array}{c}
     \varphi_\eps -\varphi_n\\
       \psi_\eps-\psi_n
\end{array}\right) 
\right\|_\oplus \right). 
\end{equation}

\subsection{The remainder term}

Defining the ``remainder'' term
$$ \Delta_n := [\tilde\Phi_n]_\oplus \left(\begin{array}{c}
     \varphi_n\\
     \psi_n
\end{array}\right)-[\tilde \Phi]_\oplus \left(\begin{array}{c}
     \varphi_n\\
     \psi_n
\end{array}\right) -\left(  [\tilde\Phi_n]_\oplus \left(\begin{array}{c}
     \varphi_\eps\\
     \psi_\eps
\end{array}\right)-[\tilde \Phi]_\oplus \left(\begin{array}{c}
     \varphi_\eps\\
     \psi_\eps
\end{array}\right) \right),$$
we have from \cref{eq:development-before-inverting-and-cross} that
\begin{multline*}
     \left\|   [\tilde\Phi_n]_\oplus \left(\begin{array}{c}
     \varphi_\eps\\
     \psi_\eps
\end{array}\right)-[\tilde \Phi]_\oplus \left(\begin{array}{c}
     \varphi_\eps\\
     \psi_\eps
\end{array}\right)+ [\mathbb{L}]_\oplus \left(\begin{array}{c}
     \varphi_\eps -\varphi_n\\
       \psi_\eps-\psi_n
\end{array}\right) \right\|_{\oplus}\\= \mathcal{O}_\PP  \left( \|\Delta_n\|_\oplus\right)+ o_\PP \left(  \left\| \left(\begin{array}{c}
     \varphi_\eps -\varphi_n\\
       \psi_\eps-\psi_n
\end{array}\right) 
\right\|_\oplus \right).
\end{multline*}
As $[\mathbb{L}]_\oplus$ is continuously invertible by \cref{lemma:invert}, this implies 
\begin{multline}\label{eq:develop-with-error}
     \left\|  [\mathbb{L}]_\oplus^{-1} \left([\tilde\Phi_n]_\oplus \left(\begin{array}{c}
     \varphi_\eps\\
     \psi_\eps
\end{array}\right)- [\tilde \Phi]_\oplus \left(\begin{array}{c}
     \varphi_\eps\\
     \psi_\eps
\end{array}\right) \right)+  \left(\begin{array}{c}
     \varphi_\eps -\varphi_n\\
       \psi_\eps-\psi_n
\end{array}\right) \right\|_{\oplus}\\= \mathcal{O}_\PP  \left( \|\Delta_n\|_\oplus\right)+ o_\PP \left(  \left\| \left(\begin{array}{c}
     \varphi_\eps -\varphi_n\\
       \psi_\eps-\psi_n
\end{array}\right) 
\right\|_\oplus \right).
\end{multline}

The key step of the proof is the following estimate for $\|\Delta_n\|_\oplus$, which will be obtained on the strength of \cref{pr:VCtool}.

\begin{lemma}\label{lemma:cross-term}
 We have
$$  \| \Delta_n\|_\oplus =o_{\mathbb{P}}\left( \| ( \varphi_n, \psi_n) - (\varphi_\eps, \psi_\eps)\|_{\oplus} + n^{-\frac{1}{2}} \right) .$$
\end{lemma}

\begin{proof}%
We use the shorthands
\begin{align}\label{eq:shorthandsXi}
    \xi_\eps(x,y)=  \langle x, y \rangle-\varphi_\eps(x)- \psi_\eps(y), \qquad \xi_n(x,y)=  \langle x, y \rangle-\varphi_n(x)- \psi_n(y).
\end{align}
Expanding the definition of $\Delta_n$, using the continuity of the quotient map, and the uniform bound~\eqref{eq:sectionsBounds}, 
we need to show that
$$ \left\|\int  ( \xi_n(\cdot,y) )_+-(\xi_\eps(\cdot,y ))_+ d(Q_n-Q)(y) \right\|_\infty \leq  o_{\mathbb{P}}\left( \| ( \varphi_n, \psi_n) - (\varphi_\eps, \psi_\eps)\|_{\oplus} + n^{-\frac{1}{2}} \right),$$
$$ \left\|\int  ( \xi_n(x,\cdot) )_+-(\xi_\eps(x,\cdot))_+ d(P_n-P)(x) \right\|_\infty \leq  o_{\mathbb{P}}\left( \| ( \varphi_n, \psi_n) - (\varphi_\eps, \psi_\eps)\|_{\oplus} + n^{-\frac{1}{2}} \right).$$
We prove the first statement, the second is analogous. Thanks to \cref{lemma:derivative-estimates-empirical-to-pop}, it suffices to show
\begin{align}\label{eq:cross-term-proof-goal1}
\left\|\int  ( \xi_n(\cdot,y) )_+-(\xi_\eps(\cdot,y ))_+ d(Q_n-Q)(y) \right\|_\infty \leq  o_{\mathbb{P}}\left( \| ( \varphi_n, \psi_n) - (\varphi_\eps, \psi_\eps)\|_{\oplus,1}\right).
\end{align}
Using the  fundamental theorem of calculus as in~\eqref{eq:fundamentalTheoremCalc}  yields
\begin{align*}
    \delta_n(x,y) &:=  (\xi_n(x,y) )_+-(\xi_\eps(x,y))_+\\
    &=(\xi_n(x,y)- \xi_\eps(x,y)) \cdot \mathcal{L}_1(\lambda\in [0,1]:  \lambda\xi_n(x,y)+(1-\lambda)\xi_\eps(x,y) \geq 0).
\end{align*}
Hence, for every $x\in \mathcal{X}$,
\begin{multline*}
     \int \delta_n(x,y) d (Q_n-Q)(y)\\
     =  \int (\xi_n(x,y)- \xi_\eps(x,y)) \cdot \mathcal{L}_1(\lambda\in [0,1]:  \lambda\xi_n(x,y)+(1-\lambda)\xi_\eps(x,y) \geq 0) d (Q_n-Q)(y) .
\end{multline*}
Defining the interpolated sections $\mathcal{S}_{\lambda, x}=\{y:\  \lambda\xi_n(x,y)+(1-\lambda)\xi_\eps(x,y) \geq 0\}$ (which are random and depend on~$n$, a fact suppressed in the notation), we can bound this as
\begin{align*}
     \left|\int \delta_n(x,y) d (Q_n-Q)(y) \right|
     & \leq \|\xi_n -\xi_\eps\|_{0,1}\sup_{\lambda\in [0, 1]} \left|\int_{\mathcal{S}_{\lambda, x}} \frac{\xi_n(x,y)- \xi_\eps(x,y)}{\|\xi_n -\xi_\eps\|_{0,1}}  d (Q_n-Q)(y)\right| \\  
     &\leq  \|\xi_n -\xi_\eps\|_{0,1}  \sup_{\|g\|_{0,1}\leq 1, \, \lambda\in [0, 1], \, x\in \mathcal{X} } \left|\int_{\mathcal{S}_{\lambda, x} \ }  g(y) d(Q_n-Q)(y)\right|.
\end{align*}
Comparing with~\eqref{eq:cross-term-proof-goal1} and noting that $\|\xi_n -\xi_\eps\|_{0,1} \leq \| ( \varphi_n, \psi_n) - (\varphi_\eps, \psi_\eps)\|_{\oplus, 1}$, it then suffices to prove
\begin{equation}\label{eq:limit-uniform-g}
    \sup_{\|g\|_{0,1}\leq 1, \, \lambda\in [0, 1],\,  x\in \mathcal{X} } \left|\int_{\mathcal{S}_{\lambda, x} \ }  g(y) d(Q_n-Q)(y)\right| =o_{\mathbb{P}} \left(1 \right).
\end{equation}
\noindent\emph{Step~1.} To that end, we first show that for any fixed bounded and measurable~$g:\mathcal{Y}\to\R$, 
\begin{equation}
    \label{eq:limit-point-wise-g-lamb}
    \sup_{\lambda\in [0, 1],\,  x\in \mathcal{X} } \left|\int_{\mathcal{S}_{\lambda, x}}  g(y) d(Q_n-Q)(y)\right| =o_{\mathbb{P}} \left(1 \right).
\end{equation}
Note that $\xi_n(x,y) \wedge \xi_\eps(x,y)\geq0$ implies $y\in\mathcal{S}_{\lambda, x}$ for all $\lambda\in[0,1]$. Moreover, if $y\in\mathcal{S}_{\lambda, x}$ for some $\lambda\in[0,1]$, we must have $\xi_n(x,y) \vee \xi_\eps(x,y)\geq0$. Recalling $\mathcal{S}_x(\cdot)$ from~\eqref{eq:SxBetaDefn}, 
this yields
 $$ \mathcal{S}_x(-\| ( \varphi_n, \psi_n) - (\varphi_\eps, \psi_\eps)\|_\oplus)\subset \mathcal{S}_{\lambda,x} \subset \mathcal{S}_x(\| ( \varphi_n, \psi_n) - (\varphi_\eps, \psi_\eps)\|_\oplus)$$
 for all $\lambda\in[0,1]$. We can now apply \cref{pr:VCtool} with
 $$ \Lambda=[0,1], \quad V_{n,\lambda,x} = W_{n,\lambda,x} = \1_{\mathcal{S}_{\lambda,x}}, \quad \beta_n=-\alpha_n=\| ( \varphi_n, \psi_n) - (\varphi_\eps, \psi_\eps)\|_\oplus
$$
 to obtain 
 $$
  \sup_{\lambda\in [0, 1],\,  x\in \mathcal{X} } \left|\int_{\mathcal{S}_{\lambda, x}}  g(y) d(Q_n-Q)(y)\right| = \mathcal{O}_{\mathbb{P}} \left(n^{-\frac{1}{2}}+ \| ( \varphi_n, \psi_n) - (\varphi_\eps, \psi_\eps)\|_\oplus \right)=o_{\mathbb{P}} \left(1 \right),
 $$
 where the last equality is due to \cref{lemma:consistency}. This completes the proof of~\eqref{eq:limit-point-wise-g-lamb}.

\vspace{.3em}

\noindent\emph{Step~2.} It remains to infer the uniform convergence~\eqref{eq:limit-uniform-g} from~\eqref{eq:limit-point-wise-g-lamb}. This follows easily from the compactness of $ \overline{\mathbb{B}}_{\|\cdot\|_{0,1}}:=\{ \|g\|_{0,1}\leq 1\}$
in $\mathcal{C}(\mathcal{Y})$. Indeed, fix $\alpha>0$. By the Arzelà--Ascoli theorem, there exist $ g_1, \dots, g_{N_\alpha}\in \overline{\mathbb{B}}_{\|\cdot\|_{0,1}}$ such that  $$\sup_{g\in \overline{\mathbb{B}}_{\|\cdot\|_{0,1}}}\inf_{k=1, \dots, N_\alpha} \|g-g_k\|_\infty \leq \frac{\alpha}{4} . $$
We can then write
\begin{align*}
     &\left| \int_{\mathcal{S}_{\lambda, x} \ }  g(y) d(Q_n-Q)(y) \right| \\
&\qquad\leq \inf_k \left(\left|  \int_{\mathcal{S}_{\lambda, x} \ } ( g(y)-g_k(y)) d(Q_n-Q)(y) \right|
+ \left|\int_{\mathcal{S}_{\lambda, x} \ }  g_k(y) d(Q_n-Q)(y) \right| \right)\\
&\qquad \leq \frac{\alpha}{2}
+ \sum_{k=1}^{N_\alpha}\left| \int_{\mathcal{S}_{\lambda, x} \ }  g_k(y) d(Q_n-Q)(y) \right|
\end{align*}
and, using also \eqref{eq:limit-point-wise-g-lamb}, conclude that
\begin{multline*}
    \mathbb{P}\left( \left| \sup_{\|g\|_{0,1}\leq 1, \, \lambda\in [0, 1],\,  x\in \mathcal{X} }\int_{\mathcal{S}_{\lambda, x} \ }  g(y) d(Q_n-Q)(y) \right| \geq \alpha \right) \\
    \leq \mathbb{P}\left( \sum_{k=1}^{N_\alpha}\left| \sup_{\lambda\in [0, 1],\,  x\in \mathcal{X} }\int_{\mathcal{S}_{\lambda, x} \ }  g_k(y) d(Q_n-Q)(y) \right|\geq \frac{\alpha}{2}  \right) \to 0
\end{multline*}
as $n\to\infty$, completing the proof.
\end{proof}

\subsection{End of the proof of \cref{th:CLT.potentials}}

Returning to the proof of \cref{th:CLT.potentials}, note that combining \eqref{eq:develop-with-error} with \cref{lemma:cross-term} yields
\begin{multline}\label{eq:develop-with-error-2}
     \left\|  [\mathbb{L}]_\oplus^{-1} \left([\tilde\Phi_n]_\oplus \left(\begin{array}{c}
     \varphi_\eps\\
     \psi_\eps
\end{array}\right)- [\tilde \Phi]_\oplus \left(\begin{array}{c}
     \varphi_\eps\\
     \psi_\eps
\end{array}\right) \right)+  \left(\begin{array}{c}
     \varphi_\eps -\varphi_n\\
       \psi_\eps-\psi_n
\end{array}\right) \right\|_{\oplus}\\= o_\PP \left(  \left\| \left(\begin{array}{c}
     \varphi_\eps -\varphi_n\\
       \psi_\eps-\psi_n
\end{array}\right) 
\right\|_\oplus + n^{-\frac12}\right).
\end{multline}
We can now complete the proof using standard arguments. 
For $(X,Y)\sim P\otimes Q$, consider the centered random processes 
\begin{align*}
\mathcal{X} \ni  x\mapsto  \mathbb{U}_Q(x)&:= \left( \langle x, Y\rangle-  \varphi_\eps(x)- \psi_\eps(Y)  \right)_+-  \int  \left( \langle x, y\rangle-  \varphi_\eps(x)- \psi_\eps(y)  \right)_+ dQ(y),\\
\mathcal{Y} \ni  y\mapsto  \mathbb{U}_P(y)&:= \left( \langle X, y\rangle-  \varphi_\eps(X)- \psi_\eps(y)  \right)_+-  \int  \left( \langle x, y\rangle-  \varphi_\eps(x)- \psi_\eps(y)  \right)_+ dP(x),
\end{align*}
As the potentials are Lipschitz by \cref{le:population.potentials}, these processes have Lipschitz sample paths. Hence, by \cite[Theorem~3.5]{Naresh.CLT.1976}, they satisfy the central limit theorem in $\mathcal{C}(\mathcal{X})$ and $\mathcal{C}(\mathcal{Y})$, respectively. As a consequence, 
$$ \sqrt{n}\left(\tilde\Phi_n \left(\begin{array}{c}
     \varphi_\eps\\
     \psi_\eps
\end{array}\right)-\tilde \Phi  \left(\begin{array}{c}
     \varphi_\eps\\
     \psi_\eps
\end{array}\right) \right) = \sqrt{n}\left( \begin{array}{c}
    \frac{1}{Q(\mathcal{S}_{(\cdot)})}\int  \left( \langle \cdot, y\rangle-  \varphi_\eps(\cdot)- \psi_\eps(y)  \right)_+ d(Q_n-Q)(y) \\  
    \frac{1}{P(\mathcal{T}_{(\cdot)})} \int  \left( \langle x, \cdot \rangle-  \varphi_\eps(x)-  \psi_\eps(\cdot )  \right)_+ d(P_n-P)(x)
\end{array} \right) $$
converges weakly in $\mathcal{C}(\mathcal{X})\times  \mathcal{C}(\mathcal{Y})$ to
$(\frac{\mathbf{G}_Q}{Q(\mathcal{S}_{(\cdot)})},\frac{\mathbf{G}_P}{P(\mathcal{T}_{(\cdot)})} 
 ),$
 where $ ({\bf G}_Q, {\bf G}_P )\in \mathcal{C}(\mathcal{X})\times  \mathcal{C}(\mathcal{Y})$ are independent, Gaussian, centered, and tight, with covariances as stated in \cref{th:CLT.potentials}. By the continuity and linearity of the quotient map $[\cdot]_\oplus$, it follows that 
\begin{align}\label{eq:proofCLTGaussian}
\sqrt{n}\left([\tilde\Phi_n]_\oplus \left(\begin{array}{c}
     \varphi_\eps\\
     \psi_\eps
\end{array}\right)-[\tilde \Phi]_\oplus  \left(\begin{array}{c}
     \varphi_\eps\\
     \psi_\eps
\end{array}\right) \right)  \overset{\mathcal{B}_\oplus}{ \rightsquigarrow} \left[  \left( \begin{array}{c}
    \frac{\mathbf{G}_Q}{Q(\mathcal{S}_{(\cdot)})}\\  
    \frac{\mathbf{G}_P}{P(\mathcal{T}_{(\cdot)})} 
\end{array} \right)\right]_\oplus.
\end{align}
Combining \eqref{eq:proofCLTGaussian} with \eqref{eq:develop-with-error-2} and the continuous mapping theorem, we conclude that
$$
\sqrt{n}\left(\begin{array}{c}
     \varphi_\eps-\varphi_n\\
       \psi_\eps - \psi_n
\end{array}\right)    \overset{\mathcal{B}_\oplus}{ \rightsquigarrow} - [\mathbb{L}]_\oplus^{-1}\left[  \left( \begin{array}{c}
    \frac{\mathbf{G}_Q}{Q(\mathcal{S}_{(\cdot)})}\\  
    \frac{\mathbf{G}_P}{P(\mathcal{T}_{(\cdot)})} 
\end{array} \right)\right]_\oplus,$$
which was the claim of \cref{th:CLT.potentials}. \qed

\section{Proofs of the CLTs for the optimal costs and couplings}\label{Section:CLT-plans-and-cost}

On the strength of the central limit theorem for the potentials (\cref{th:CLT.potentials}) and the statistical gradient estimate (\cref{lemma:derivative-estimates-empirical-to-pop}), we can now easily derive the central limit theorem for the optimal costs.

\begin{proof}[Proof of \cref{th:CLT.cost}]
The optimality of the population potentials $(f_\eps,g_\eps)$ for the population dual problem~\eqref{DQOT} yields
\begin{align}
    \begin{split}\label{Upper-bound-Cost}
        {\rm QOT}_\eps(P_n,Q_n)&- {\rm QOT}_\eps(P,Q)\\
    &\geq  \int f_\eps(x) d(P_n-P)(x) +\int g_\eps(y) d(Q_n-Q)(y)\\ &\qquad - \frac{1}{2 \eps} \int \left( f_\eps(x)+g_\eps(y)- \frac{\|x-y\|^2}{2}\right)_+^2 d (P_n\otimes Q_n- P\otimes Q)(x,y) 
    \end{split}
\end{align}
whereas the optimality of the empirical potentials $(f_n,g_n)$ for the empirical dual yields
\begin{align}\label{Lower-bound-Cost}
    \begin{split}
        {\rm QOT}_\eps(P_n,Q_n)&- {\rm QOT}_\eps(P,Q)\\
    &\leq  \int f_n(x) d(P_n-P)(x) +\int g_n(y) d(Q_n-Q)(y)\\ &\qquad - \frac{1}{2 \eps} \int \left( f_n(x)+g_n(y)- \frac{\|x-y\|^2}{2}\right)_+^2 d (P_n\otimes Q_n- P\otimes Q)(x,y) .
    \end{split}
\end{align}
Note that the right-hand side of \eqref{Upper-bound-Cost} gives the limit described in \cref{th:CLT.cost}. Hence, it suffices to show that the difference between the right-hand sides of \eqref{Upper-bound-Cost} and \eqref{Lower-bound-Cost} behaves as $o_\PP(n^{-\frac{1}{2}})$; that is,
\begin{align}\label{eq:proofCostCLTclaim}
  \int (\tau_\eps-\tau_n)  d (P_n\otimes Q_n- P\otimes Q) = o_\PP(n^{-\frac{1}{2}})
\end{align}
for
\begin{align*}
\tau_\eps(x,y)&:=f_\eps(x)+ g_\eps(y)-\frac{1}{2 \eps}  \left( f_\eps(x)+g_\eps(y)- \frac{\|x-y\|^2}{2}\right)_+^2, \\
\tau_n(x,y)&:=f_n(x)+ g_n(y)-\frac{1}{2 \eps}  \left( f_n(x)+g_n(y)- \frac{\|x-y\|^2}{2}\right)_+^2.
\end{align*}
To see this, we first write 
\begin{align}
   \left| \int (\tau_\eps-\tau_n)  d (P_n\otimes Q_n- P\otimes Q) \right|&\leq \|\tau_\eps-\tau_n\|_{0,1}  \sup_{\|h\|_{0,1}\leq 1}\int h  d (P_n\otimes Q_n- P\otimes Q) \nonumber\\
   &=o_\PP\left(\|\tau_\eps-\tau_n\|_{0,1} \right) \label{eq:proofCostsCLT2},
\end{align}
where the second estimate holds because the unit ball in $\mathcal{C}^{0,1}
(\mathcal{X}\times \mathcal{Y})$ is a Glivenko--Cantelli class. Noting that the function $(\cdot)_+^2/2$ occurring in the definitions of $\tau_\eps$ and $\tau_n$ has Lipschitz-continuous derivative $(\cdot)_+$, our gradient estimate in \cref{lemma:derivative-estimates-empirical-to-pop} yields 
\begin{align*}
   o_\PP\left(\|\tau_\eps-\tau_n\|_{0,1} \right) = o_{\mathbb{P}} \left(n^{-\frac{1}{2}}+ \| ( f_n, g_n) - (f_\eps, g_\eps)\|_\oplus \right). %
\end{align*}
As we already know from  \cref{th:CLT.potentials} that $\| ( f_n, g_n) - (f_\eps, g_\eps)\|_\oplus=\mathcal{O}_{\mathbb{P}}(n^{-\frac{1}{2}})$, this completes the proof of~\eqref{eq:proofCostCLTclaim}.
\end{proof}

Lastly, we prove the central limit theorem for the optimal couplings.

\begin{proof}[Proof of \cref{th:CLT.plans}]
Note that 
$\int \eta d(\pi_n - \pi)= \int \bar{\eta} d(\pi_n - \pi)$ and define 
$$  \mathcal{E}_n:=  \int  \bar{\eta}(x,y)  \left( \xi_n(x,y)\right)_+ d(P_n\otimes Q_n)(x,y)-   \int  \bar{\eta}(x,y)  \left( \xi_\eps(x,y)\right)_+ d(P\otimes Q)(x,y).$$
Recalling~\eqref{eq:primal-dual}, we need to show
$\sqrt{n} \mathcal{E}_n/\eps\xrightarrow{w}  N\left(0, \sigma^2(\eta)/\eps^2\right)$. To that end, we decompose
\begin{align*}
    \mathcal{E}_n= A_n + B_n + C_n
\end{align*}
where
\begin{align*}
    A_n&:= \int  \bar{\eta}(x,y) \{ \left( \xi_n(x,y)\right)_+- (\xi_\eps(x,y))_+ \} d(P_n\otimes Q_n-P\otimes Q)(x,y),\\
    B_n&:= \int  \bar{\eta}(x,y) \{ \left( \xi_n(x,y)\right)_+- (\xi_\eps(x,y))_+ \} d(P\otimes Q)(x,y),\\
    C_n&:= \int  \bar{\eta}(x,y) (\xi_\eps(x,y))_+ d(P_n\otimes Q_n-P\otimes Q)(x,y).
\end{align*}
\noindent\emph{Step~1.} We first address the key part of the proof, which is to show
\begin{align}\label{eq:proofplansCLT0}
A_n= o_\PP\left(n^{-\frac{1}{2}}\right).
\end{align}
Using the  fundamental theorem of calculus as in~\eqref{eq:fundamentalTheoremCalc}  yields  $$  ( \xi_n(x,y))_+- (\xi_\eps(x,y))_+  = (\xi_n(x,y)- \xi_\eps(x,y)) V_n(x,y)$$
where $V_n(x,y):=\mathcal{L}_1(\{\lambda\in [0,1]: \lambda \xi_n(x,y) + (1-\lambda)\xi_\eps(x,y)\geq 0\})$.
We can therefore write 
\begin{align*}
A_n
&= \|\xi_n- \xi_\eps \|_{0,1}\int  \bar{\eta}(x,y) \frac{(\xi_n- \xi_\eps)(x,y)}{\|\xi_n- \xi_\eps \|_{0,1}} V_n(x,y) d(P_n\otimes Q_n-P\otimes Q)(x,y)\\
&\leq \|\xi_n- \xi_\eps \|_{0,1} \sup_{\| h\|_{0,1}\leq 1} \int \bar{\eta}(x,y) h(x,y)   V_n(x,y)   d(P_n\otimes Q_n-P\otimes Q)(x,y).
\end{align*}
An analogous lower bound holds with an infimum; we only discuss the supremum. 
Once again, our gradient estimate in \cref{lemma:derivative-estimates-empirical-to-pop} and \cref{th:CLT.potentials} yield
\begin{align*}
   o_\PP\left(\|\xi_n- \xi_\eps \|_{0,1} \right) = o_{\mathbb{P}} \left(n^{-\frac{1}{2}}+ \| ( \varphi_n, \psi_n) - (\varphi_\eps, \psi_\eps)\|_\oplus \right)=\mathcal{O}_{\mathbb{P}}\left(n^{-\frac{1}{2}}\right).
\end{align*}
Thus, it remains to show that 
\begin{align}\label{eq:proofplansCLT1}
\sup_{\| h\|_{0,1}\leq 1} \int \bar{\eta}(x,y) h(x,y)   V_n(x,y)   d(P_n\otimes Q_n-P\otimes Q)(x,y)=o_{\mathbb{P}} \left(1 \right).
\end{align}
We first show, for any bounded measurable $h:\mathcal{X}\times\mathcal{Y}\to\R$, that
\begin{align}\label{eq:proofplansCLT2}
\int \bar{\eta}(x,y) h(x,y)   V_n(x,y)   d(P_n\otimes Q_n-P\otimes Q)(x,y)=o_{\mathbb{P}} \left(1 \right).
\end{align}
Let $\delta_n= \| \xi_n- \xi_\eps\|_\infty$ and recall from~\eqref{eq:MBetaDefn} the notation $\mathcal{M}(\cdot)$. On the set $\mathcal{M}(-\delta_n)=\{\xi_\eps\geq \delta_n\}$, we have both $\xi_\eps\geq0$ and $\xi_n\geq0$, so that $V_n=1$. In view of $0\leq V_n\leq1$, it follows that $\1_{\mathcal{M}(-\delta_n)}\leq V_n$. 
On the set $\{V_n>0\}$, we have either $\xi_\eps\geq0$ or $\xi_n\geq0$, and hence $\xi_\eps+ \| \xi_n- \xi_\eps\|_\infty \geq 0$. Thus, $\{V_n>0\}\subset \{\xi_\eps\geq -\delta_n\}=\mathcal{M}(\delta_n)$, showing that $V_n\leq \1_{\mathcal{M}(\delta_n)}$. In summary,
$$
\1_{\mathcal{M}(-\delta_n)}\leq V_n \leq \mathcal{M}(\delta_n).
$$
Recalling that $\delta_n= o_{\mathbb{P}} \left(1 \right)$ by \cref{th:CLT.potentials}, we can thus apply \cref{pr:VCtoolProduct} with
$$ g=\bar{\eta}h, \quad \Lambda=\emptyset, \quad V_{n,\lambda}=W_{n,\lambda} = V_n, \quad \beta_n=-\alpha_n=\delta_n
$$
to conclude that 
\begin{align*}
\int \bar{\eta}(x,y) h(x,y)   V_n(x,y)   d(P_n\otimes Q_n-P\otimes Q)(x,y)=\mathcal{O}_{\mathbb{P}}\left(n^{-\frac12} + \delta_n \right) = o_{\mathbb{P}} \left(1 \right),
\end{align*}
which is~\eqref{eq:proofplansCLT2}. Finally, \eqref{eq:proofplansCLT1} follows from \eqref{eq:proofplansCLT2} and the compactness of $\{ \|h\|_{0,1}\leq 1\}$, exactly as in Step~2 in the proof of~\cref{lemma:cross-term}. This completes the proof of~\eqref{eq:proofplansCLT0}.

\vspace{.3em}

\noindent\emph{Step~2.} It remains to show that $\sqrt{n} (B_n+C_n)/\eps\xrightarrow{w}  N\left(0, \sigma^2(\eta)/\eps^2\right)$. The real-valued function
$$  \Banach\ni (f,g)\mapsto   \int   \bar{\eta}(x,y) \left( f(x)+g(x)-\frac{\|x-y\|^2}{2}\right)_+  d(P\otimes Q)(x,y)  $$
is Fréchet differentiable at the population potentials $(f_\eps,g_\eps)$ with derivative 
$$  \Banach\ni (f,g)\mapsto   \int_{\xi_\eps \geq 0}   \bar{\eta}(x,y)  (f(x)+ g(y)) d(P\otimes Q)(x,y)  .$$
The proof of this fact uses the continuity of $(f_\eps,g_\eps)$ and that $(P\otimes Q)(\partial\{\xi_\eps \geq 0\})=0$ as a consequence of Fubini's theorem and $Q(\partial\mathcal{S}_x)=0$; we omit the details as the argument is similar to (but simpler than) the proof of \cref{lemma:frechet}.
As a consequence,
$$   B_{n}=  \int_{\xi_\eps\geq 0 } \bar{\eta}(x,y) \left(  \xi_n(x,y) -\xi_\eps(x,y) \right)  d(P\otimes Q)(x,y) + o_{\PP}\left(\|   \xi_n -\xi_\eps\|_\infty \right),$$
and now \cref{th:CLT.potentials} yields that $B_{n}=B'_{n}+o_{\PP}(n^{-\frac{1}{2}})$ for
$$B'_{n}=\frac{1}{n^2} \sum_{i,j=1}^n \int_{\xi_\eps\geq 0} \bigg\{   \oplus\left( \mathbb{L}^{-1}\left[\left( \begin{array}{c}
           \frac{ (\xi_\eps( \cdot ,Y_j))_+ - \int (\xi_\eps( \cdot ,y'))_+dQ (y') } {Q(\mathcal{S}_{(\cdot)})}   \\
            \frac{ (\xi_\eps(X_i, \cdot ))_+-\int (\xi_\eps(x', \cdot ))_+dP(x') }{P(\mathcal{T}_{(\cdot)})}  
        \end{array}\right)\right]_{\oplus} \right)\bar{\eta} \bigg\} d(P\otimes Q),$$
where we recall that $X_i,Y_j$ denote the samples defining~$P_n$ and~$Q_n$. We note that both $B_n'$ and 
\begin{align*}
   C_n= \frac{1}{n^2}\sum_{i,j=1}^n \bar{\eta}(X_i,Y_j)(\xi_\eps(X_i,Y_j))_+ \,-\,\int \bar{\eta}(\xi_\eps)_+  d(P\otimes Q)
\end{align*}
are centered (w.r.t.~$\PP$) and have finite variance. In summary,  
$$\int \bar\eta d(\pi_n - \pi)=B'_n+ C_n + o_\PP(n^{-\frac{1}{2}})$$ is a $U$-statistic with finite variance up $o_\PP(n^{-\frac{1}{2}})$ terms, so that the central limit theorem for  $U$-statistics (cf.~\cite[Theorem~12.6]{Vaart.1998})  gives the desired result.
\end{proof}

\appendix

\section{Omitted Proofs}\label{se:omittedProofs}

This section collects the proofs for the statements in \cref{se:SetupBackground}, which are either known or follow easily from known results.

\begin{proof}[Proof of \cref{le:population.potentials}]
    Item (i), the first half of (ii), item (iv), and the first part of (v) up to~\eqref{eq:primal-dual}, 
    can all be found in \cite{Nutz.24}. The convexity of $\varphi_\eps$ and $\psi_\eps$ and the formula \eqref{eq:population-gradient-formula} are shown in \cite{WieselXu.24} and \cite{GonzalezSanzNutz2024.Scalar} (or see the proof of \cref{le:empirical.potentials} below). 

    For any $x\in \mathcal{X}$, the set $ \mathcal{S}_x=\{y\in\mathcal{Y}:\, \varphi_\eps(x)+\psi_\eps(y)-\langle x,y \rangle \leq 0\}$ is convex as a sublevel set of the convex function $y\mapsto \varphi_\eps(x)+\psi_\eps(y)-\langle x,y \rangle$. Clearly 
$$
\partial{\mathcal{S}}_x\subset\{y\in\mathcal{Y}:\,\varphi_\eps(x)+\psi_\eps(y)-\langle x,y \rangle=0\}.
$$
For the reverse inclusion, it suffices to observe that a point $y$ with $\varphi_\eps(x)+\psi_\eps(y)-\langle x,y \rangle=0$ cannot be a  minimum of the convex function $y\mapsto \varphi_\eps(x)+\psi_\eps(y)-\langle x,y \rangle$, as otherwise it would follow that this function is nonnegative, contradicting~\eqref{population_optimality}. 

As $\partial{\mathcal{S}}_x$ is the boundary of a convex set and $Q\ll \mathcal{L}_d$, we have that $Q(\partial{\mathcal{S}_x})=0$ and hence $Q(\{y\in\mathcal{Y}:\, \varphi_\eps(x)+\psi_\eps(y)-\langle x,y \rangle = 0\})=0$. By the continuity of $\varphi_\eps(x)+\psi_\eps(y)-\langle x,y \rangle$, it follows that 
$
  {\bf 1}_{\mathcal{S}_{x_n}}\to {\bf 1}_{\mathcal{S}_{x}}
$ $Q$-a.s.\ for $x_n\to x$. In particular, $x\mapsto Q(\mathcal{S}_{x})$ is continuous and, in view of its formula~\eqref{eq:population-gradient-formula}, $\nabla \varphi_\eps$ is continuous.

As before, $Q({\mathcal{S}}_x)>0$ by \eqref{population_optimality}, and hence continuity of $x\mapsto Q(\mathcal{S}_{x})$ also implies the uniform bound~\eqref{eq:sectionsBounds}. (A different proof of~\eqref{eq:sectionsBounds} can be found in \cite{BayraktarEckstein.2025.BJ}.) The expression for $\spt\pi_\eps$ in~(iv) follows from \eqref{eq:primal-dual} and the aforementioned fact that $Q(\partial{\mathcal{S}_x})=0$.
\end{proof}

\begin{proof}[Proof of \cref{le:empirical.potentials}]
    The existence of the potentials and their continuous extension to $\R^d$ 
    satisfying~\eqref{empirical_optimality} can again be found in \cite{Nutz.24}. We detail the proof of the convexity and differentiability properties as \cite{WieselXu.24} and  \cite{GonzalezSanzNutz2024.Scalar} do not state them in the discrete setting; however, we remark that the proof is similar.

We first verify that $\varphi_n$ is convex. Recall from~\eqref{empirical_optimality} that 
\begin{equation}
    \label{eq:definitionphi-proof}
    \eps=\int  \left(  \langle x,y \rangle -\varphi_n(x)-\psi_n(y)\right)_+ dQ_n(y) \quad\mbox{for all } x\in \mathcal{X}.
\end{equation}
For every  $x,x'\in \mathcal{X}$ and every $\lambda\in (0,1)$, convexity of $(\cdot)_+$ yields
\begin{align*}
    \int & \left(  \langle \lambda x+(1-\lambda) x',y \rangle -\lambda \varphi_n(x)- (1-\lambda) \varphi_n(x') -\psi_n(y)\right)_+ dQ_n(y) \\
    & \leq \lambda\int  \left(  \langle  x ,y \rangle - \varphi_n(x)-\psi_n(y)\right)_+ dQ_n(y) + (1-\lambda)\int  \left(  \langle  x' ,y \rangle - \varphi_n(x')-\psi_n(y)\right)_+ dQ_n(y)
\end{align*}
and the right-hand side equals $\eps$ by~\eqref{eq:definitionphi-proof}. On the other hand, \eqref{eq:definitionphi-proof} at the point $\lambda x+(1-\lambda) x'$ states that
$$ \eps= \int \left(  \langle \lambda x+(1-\lambda) x',y \rangle - \varphi_n(\lambda x+(1-\lambda) x')  -\psi_n(y)\right)_+ dQ_n(y).
 $$
Together, we have
\begin{multline*}
    \int \left(  \langle \lambda x+(1-\lambda) x',y \rangle -\lambda \varphi_n(x)- (1-\lambda) \varphi_n(x') -\psi_n(y)\right)_+ dQ_n(y)\\
    \leq \int \left(  \langle \lambda x+(1-\lambda) x',y \rangle - \varphi_n(\lambda x+(1-\lambda) x') -\psi_n(y)\right)_+ dQ_n(y),
\end{multline*}
and moreover the right-hand side equals $\eps>0$. It follows that the non-increasing function
$$ t\mapsto \int \left(  \langle \lambda x+(1-\lambda) x',y \rangle - t -\psi_n(y)\right)_+ dQ_n(y)$$
is strictly decreasing for $t$ in a neighborhood of $\varphi_n(\lambda x+(1-\lambda) x')$, so that we can conclude
$$ \varphi_n(\lambda x+(1-\lambda) x') \leq \lambda \varphi_n(x)+ (1-\lambda) \varphi_n(x').$$
In other words, $\varphi_n$ is convex. 

Next, we observe that \eqref{eq:definitionphi-proof} and the inequality $ (a)_+-(b)_+ \leq \mathbb{I}_{\{a\geq 0\}} (a-b)$, applied with
$a=(  \langle x,y \rangle -\varphi_n(x)-\psi_n(y))_+$ and  $b=\left(  \langle x',y \rangle -\varphi_n(x')-\psi_n(y)\right)_+$, imply that for $x,x'\in\mathcal{X}$,
\begin{eqnarray*}
\lefteqn{0=\int  \left(  \langle x,y \rangle -\varphi_n(x)-\psi_n(y)\right)_+ dQ_n(y)-\int  \left(  \langle x',y \rangle -\varphi_n(x')-\psi_n(y)\right)_+ dQ_n(y)}\hspace*{5cm}\\
&\leq&\int_{\hat{\mathcal{S}}_x}    \langle x-x',y \rangle dQ_n(y)-Q_n(\hat{\mathcal{S}}_x)(\varphi_n(x)-\varphi_n(x')).
\end{eqnarray*}
Note that $Q_n(\hat{\mathcal{S}}_x)$ cannot vanish, by \eqref{eq:definitionphi-proof}. Rearranging then gives
$$\varphi_n(x')- \varphi_n(x)\geq \Big\langle x'-x, \textstyle \frac{\int_{\hat{\mathcal{S}}_x} y dQ_n(y)}{Q_n(\hat{\mathcal{S}}_x)} \Big\rangle,$$
showing that $\frac{\int_{\hat{\mathcal{S}}_x} y dQ_n(y)}{Q_n(\hat{\mathcal{S}}_x)} $ is a subgradient of $\varphi_n$ at $x$. On the other hand, the convex function $ \varphi_n$ is $\mathcal{L}_d$-a.e.\ differentiable by Rademacher's theorem. Together, we conclude~\eqref{eq:emp-gradient-formula}. This formula also shows that $\|\nabla  \varphi_n(x)\|\leq \|\mathcal{Y}\|_\infty$ a.e., and hence that $  \varphi_n$ is $\|\mathcal{Y}\|_\infty$-Lipschitz.
\end{proof}

\begin{proof}[Proof of \cref{lemma:consistency}]
Almost surely, $P_n$ and $Q_n$ are probability measures supported in $\mathcal{X}$ and $\mathcal{Y}$ which converge weakly to $P$ and $Q$, respectively. We fix such a realization and show that any potentials $(\varphi_n,\psi_n)$ associated with $P_n$ and $Q_n$  converge to $(\varphi_\eps, \psi_\eps)$ in $\mathcal{B}_\oplus$. %
As we deal with equivalence classes, we may choose representatives of  $(\varphi_n, \psi_n)$  with 
$\varphi_n(x_0)=0$ 
for some fixed $x_0\in \mathcal{X}$. We recall from \cref{le:empirical.potentials} that $(\varphi_n,\psi_n)$ are Lipschitz with a constant depending only on $\mathcal{X}$ and $\mathcal{Y}$. Observe that given any $y\in\mathcal{Y}$, \eqref{empirical_optimality} implies that $\langle x,y\rangle -\varphi_n(x)-\psi_n(y)>0$ for some $x\in \mathcal{X}$. Using $\varphi_n(x_0)=0$ and boundedness of the supports, we deduce a uniform upper bound $\psi_n(y)\leq \varphi_n(x)-\langle x,y\rangle \leq C$, which by the Lipschitz continuity also implies a uniform lower bound. Moreover, $\langle x,y\rangle -\varphi_n(x)-\psi_n(y)>0$ now yields another upper bound $\varphi_n(x)\leq C$, which again also implies a  lower bound. We conclude that $(\varphi_n,\psi_n)$ are uniformly bounded and equicontinuous, so that by the Arzelà--Ascoli theorem, a subsequence $\{(\varphi_{n_k},\psi_{n_k})\}_k$ converges uniformly to a limit $(\varphi_\infty,\psi_\infty)$.
Note that $(x,y)\mapsto \langle x, y\rangle -\varphi_n(x)- \psi_n(y)$ is uniformly Lipschitz with some constant $L$ and hence
$$  \left\vert \int \left( \langle x, y\rangle -\varphi_{n_k}(x)- \psi_{n_k}(y) \right)_+ d(Q_{n_k}-Q)(y)  \right\vert  \leq L\cdot \sup_{\|f\|_{0,1}\leq 1} \int f  d(Q_{n_k}-Q) \to 0.$$
Using this observation and the empirical first-order condition~\eqref{empirical_optimality}, passing to the limit shows that $(\varphi_\infty,\psi_\infty)$ satisfies the population first-order condition~\eqref{population_optimality}. Thus $(\varphi_\infty,\psi_\infty)$ are population potentials and the claim follows by their uniqueness. 
\end{proof}

\bibliographystyle{abbrv}
\bibliography{biblio}

\begin{thebibliography}{10}

\bibitem{balakrishnan.et.al.2025.survey}
S.~Balakrishnan, T.~Manole, and L.~Wasserman.
\newblock Statistical inference for optimal transport maps: Recent advances and perspectives.
\newblock {\em arXiv:2506.19025}, 2025.

\bibitem{BayraktarEckstein.2025.BJ}
E.~Bayraktar, S.~Eckstein, and X.~Zhang.
\newblock Stability and sample complexity of divergence regularized optimal transport.
\newblock {\em Bernoulli}, 31(1):213--239, 2025.

\bibitem{blondel18quadratic}
M.~Blondel, V.~Seguy, and A.~Rolet.
\newblock Smooth and sparse optimal transport.
\newblock In {\em Proceedings of the Twenty-First International Conference on Artificial Intelligence and Statistics}, volume~84 of {\em Proceedings of Machine Learning Research}, pages 880--889, 2018.

\bibitem{Caffarelli.1996.AnnalsOfMath}
L.~A. Caffarelli.
\newblock Boundary regularity of maps with convex potentials--ii.
\newblock {\em The Annals of Mathematics}, 144(3):453, 1996.

\bibitem{ChewiNilesWeedRigollet.25}
S.~Chewi, J.~Niles-Weed, and P.~Rigollet.
\newblock {\em Statistical optimal transport}, volume 2364 of {\em Lecture Notes in Mathematics}.
\newblock Springer, Cham, 2025.

\bibitem{Clarke}
F.~H. Clarke.
\newblock {\em Optimization and Nonsmooth Analysis}.
\newblock Society for Industrial and Applied Mathematics, 1990.

\bibitem{Cuturi.2013.Neurips}
M.~Cuturi.
\newblock Sinkhorn distances: Lightspeed computation of optimal transport.
\newblock In C.~Burges, L.~Bottou, M.~Welling, Z.~Ghahramani, and K.~Weinberger, editors, {\em Advances in Neural Information Processing Systems}, volume~26. Curran Associates, Inc., 2013.

\bibitem{delbarrio.et.al.2025.survey}
E.~del Barrio, A.~González-Sanz, J.-M. Loubes, and D.~Rodríguez-Vítores.
\newblock Distributional limit theory for optimal transport.
\newblock {\em arXiv:2505.19104}, 2025.

\bibitem{delBarrioEtAl.2023.SIMODS}
E.~del Barrio, A.~G. Sanz, J.-M. Loubes, and J.~Niles-Weed.
\newblock An improved central limit theorem and fast convergence rates for entropic transportation costs.
\newblock {\em SIAM Journal on Mathematics of Data Science}, 5(3):639--669, 2023.

\bibitem{DelfourZolesio}
M.~C. Delfour and J.-P. Zol\'esio.
\newblock Shape analysis via oriented distance functions.
\newblock {\em J. Funct. Anal.}, 123(1):129--201, 1994.

\bibitem{EcksteinNutz.22}
S.~Eckstein and M.~Nutz.
\newblock Convergence rates for regularized optimal transport via quantization.
\newblock {\em Math. Oper. Res.}, 49(2):1223--1240, 2024.

\bibitem{EssidSolomon.18}
M.~Essid and J.~Solomon.
\newblock Quadratically regularized optimal transport on graphs.
\newblock {\em SIAM J. Sci. Comput.}, 40(4):A1961--A1986, 2018.

\bibitem{GarrizmolinaElAl.2024}
A.~Garriz-Molina, A.~González-Sanz, and G.~Mordant.
\newblock Infinitesimal behavior of quadratically regularized optimal transport and its relation with the porous medium equation.
\newblock {\em arXiv:2407.21528}, 2024.

\bibitem{genevay.2019.PMLR}
A.~Genevay, L.~Chizat, F.~Bach, M.~Cuturi, and G.~Peyr\'{e}.
\newblock Sample complexity of {S}inkhorn divergences.
\newblock In K.~Chaudhuri and M.~Sugiyama, editors, {\em Proceedings of the Twenty-Second International Conference on Artificial Intelligence and Statistics}, volume~89 of {\em Proceedings of Machine Learning Research}, pages 1574--1583. PMLR, 2019.

\bibitem{goldfeld.2024.statisticalinferenceregularizedoptimal}
Z.~Goldfeld, K.~Kato, G.~Rioux, and R.~Sadhu.
\newblock Statistical inference with regularized optimal transport.
\newblock {\em Inf. Inference}, 13(1):Paper No. 13, 68, 2024.

\bibitem{GonzalezSanzNutz.24b}
A.~Gonz{\'a}lez-Sanz and M.~Nutz.
\newblock Sparsity of quadratically regularized optimal transport: Scalar case.
\newblock {\em arXiv:2410.03353}, 2024.

\bibitem{GonzalezSanzNutzRiveros.25gradDesc}
A.~Gonz\'alez-Sanz, M.~Nutz, and A.~Riveros~Valdevenito.
\newblock Linear convergence of gradient descent for quadratically regularized optimal transport.
\newblock {\em arXiv:2509.08547v2}, 2025.

\bibitem{gonzalezsanz.2025.sparseregularizedoptimaltransport}
A.~González-Sanz, S.~Eckstein, and M.~Nutz.
\newblock Sparse regularized optimal transport without curse of dimensionality.
\newblock {\em arXiv:2505.04721}, 2025.

\bibitem{GonzalezSanz.2023.Beyond}
A.~González-Sanz and S.~Hundrieser.
\newblock Weak limits for empirical entropic optimal transport: Beyond smooth costs.
\newblock {\em arXiv:2305.09745}, 2023.

\bibitem{GonzalezSanz.2024.weaklimits}
A.~González-Sanz, J.-M. Loubes, and J.~Niles-Weed.
\newblock Weak limits of entropy regularized optimal transport; potentials, plans and divergences.
\newblock {\em arXiv:2207.07427}, 2024.

\bibitem{GonzalezSanzNutz2024.Scalar}
A.~González-Sanz and M.~Nutz.
\newblock Sparsity of quadratically regularized optimal transport: Scalar case.
\newblock {\em arXiv:2410.03353}, 2024.

\bibitem{Naresh.CLT.1976}
N.~C. Jain.
\newblock Central limit theorem in a {B}anach space.
\newblock In A.~Beck, editor, {\em Probability in Banach Spaces}, pages 113--130, Berlin, Heidelberg, 1976. Springer.

\bibitem{LiGenevayYurochkinSolomon.2020.NIPS}
L.~Li, A.~Genevay, M.~Yurochkin, and J.~M. Solomon.
\newblock Continuous regularized {W}asserstein barycenters.
\newblock In H.~Larochelle, M.~Ranzato, R.~Hadsell, M.~Balcan, and H.~Lin, editors, {\em Advances in Neural Information Processing Systems}, volume~33, pages 17755--17765. Curran Associates, Inc., 2020.

\bibitem{Lorenz.2019}
D.~A. Lorenz, P.~Manns, and C.~Meyer.
\newblock Quadratically regularized optimal transport.
\newblock {\em Applied Mathematics and Optimization}, 83(3):1919–1949, 2019.

\bibitem{MenaWeed.2019.Nips}
G.~Mena and J.~Niles-Weed.
\newblock Statistical bounds for entropic optimal transport: sample complexity and the central limit theorem.
\newblock In H.~Wallach, H.~Larochelle, A.~Beygelzimer, F.~d\textquotesingle Alch\'{e}-Buc, E.~Fox, and R.~Garnett, editors, {\em Advances in Neural Information Processing Systems}, volume~32. Curran Associates, Inc., 2019.

\bibitem{Muzellec.2017.AAAI}
B.~Muzellec, R.~Nock, G.~Patrini, and F.~Nielsen.
\newblock Tsallis regularized optimal transport and ecological inference.
\newblock {\em Proceedings of the AAAI Conference on Artificial Intelligence}, 31(1), 2017.

\bibitem{Nutz.24}
M.~Nutz.
\newblock Quadratically regularized optimal transport: Existence and multiplicity of potentials.
\newblock {\em SIAM Journal on Mathematical Analysis}, 57(3):2622--2649, 2025.

\bibitem{rigollet2022samplecomplexityentropicoptimal}
P.~Rigollet and A.~J. Stromme.
\newblock On the sample complexity of entropic optimal transport.
\newblock {\em Ann. Statist.}, 53(1):61--90, 2025.

\bibitem{Urbas.1997}
J.~Urbas.
\newblock On the second boundary value problem for equations of {M}onge-{A}mpère type.
\newblock {\em Journal für die reine und angewandte Mathematik}, 487:115--124, 1997.

\bibitem{Vaart.1998}
A.~W. van~der Vaart.
\newblock {\em Asymptotic Statistics}.
\newblock Cambridge University Press, 1998.

\bibitem{vanderVaart.1996}
A.~W. van~der Vaart and J.~A. Wellner.
\newblock {\em Weak Convergence and Empirical Processes}.
\newblock Springer, 1996.

\bibitem{Wang-PrAmMaSo-95}
X.-J. Wang.
\newblock Some counterexamples to the regularity of {M}onge-{A}mpère equations.
\newblock {\em Proceedings of the American Mathematical Society}, 123(3):841--845, 1995.

\bibitem{WieselXu.24}
J.~Wiesel and X.~Xu.
\newblock Sparsity of quadratically regularized optimal transport: Bounds on concentration and bias.
\newblock {\em Preprint arXiv:2410.03425}, 2024.

\bibitem{zhang.2023.manifoldlearningsparseregularised}
S.~Zhang, G.~Mordant, T.~Matsumoto, and G.~Schiebinger.
\newblock Manifold learning with sparse regularised optimal transport.
\newblock {\em arXiv:2307.09816}, 2023.

\end{thebibliography}

\end{document}